\documentclass{article}

\usepackage{amsfonts}
\usepackage{amsmath}
\usepackage{amssymb}
\usepackage{wasysym}
\usepackage{yfonts}  
\usepackage{amsthm}
\usepackage{amscd}   
\usepackage[all]{xypic}
\usepackage{mathrsfs}
\usepackage{indentfirst}
\usepackage{bbold}
\usepackage{amsxtra}
\usepackage{amsbsy}
\usepackage{amstext}
\usepackage{amsopn}
\usepackage{enumerate}

\newenvironment{enumerate*}{
\begin{enumerate}[{\rm (i)}]
  \setlength{\itemsep}{3.5pt}
  \setlength{\parskip}{0pt}
}{\end{enumerate}}

\newenvironment{enumerate!}{
\begin{enumerate}[{\rm I.}]
  \setlength{\itemsep}{3.5pt}
  \setlength{\parskip}{0pt}
}{\end{enumerate}}

\newenvironment{enumerateabc}{
\begin{enumerate}[{\rm (a)}]
  \setlength{\itemsep}{3.5pt}
  \setlength{\parskip}{0pt}
}{\end{enumerate}}

\newenvironment{enumeratenum}{
\begin{enumerate}[{\rm (1)}]
  \setlength{\itemsep}{3.5pt}
  \setlength{\parskip}{0pt}
}{\end{enumerate}}

\CompileMatrices

\renewcommand*{\thefootnote}{\arabic{footnote}}

\newtheorem{theorem}{Theorem}[section]
\newtheorem{lemma}[theorem]{Lemma}
\newtheorem{proposition}[theorem]{Proposition}
\newtheorem{corollary}[theorem]{Corollary} 
\newtheorem{definition}[theorem]{Definition}
\newtheorem{example}[theorem]{Example}
\newtheorem{remark}[theorem]{Remark}

\newtheorem*{acknowledgement}{Acknowledgement}
\newtheorem*{acknowledgements}{Acknowledgements}

\newtheorem{openquestion1}[theorem]{Open Question}
\newtheorem{openquestion2}[theorem]{Open Question}

\newtheorem{question1}[theorem]{Question}
\newtheorem{question2}[theorem]{Question}
\newtheorem{construct}[theorem]{Construction}

\numberwithin{equation}{section}

\usepackage{titlesec}

\titleformat*{\section}{\large \bfseries}
\titleformat*{\subsection}{\it}
\titleformat*{\subsubsection}{\large\bfseries}
\titleformat*{\paragraph}{\large\bfseries}
\titleformat*{\subparagraph}{\large\bfseries}

\begin{document}

\title{Virtually torsion-free covers of minimax groups} 

\author{
{\sc Peter Kropholler}\\
{\sc Karl Lorensen}
}

\maketitle
\begin{abstract} We prove that every finitely generated, virtually solvable minimax group can be expressed as a homomorphic image of a virtually torsion-free, virtually solvable minimax group.  
This result enables us to generalize a theorem of Ch. Pittet and L. Saloff-Coste about random walks on finitely generated, virtually solvable minimax groups. Moreover, the paper identifies properties, such as the derived length and the nilpotency class of the Fitting subgroup, that are preserved in the covering process. Finally, we determine exactly which infinitely generated, virtually solvable minimax groups also possess this type of cover.
\vspace{20pt}

\centerline{\it Couvertures virtuellement sans torsion de groupes minimax}
\vspace{5pt}
\centerline{\bf R\'esum\'e}
\vspace{5pt}

{\it Nous prouvons que tout groupe de type fini minimax et virtuellement r\'esoluble peut \^etre exprim\'e comme l'image homomorphe d'un groupe minimax virtuellement r\'esoluble et virtuellement sans torsion. Ce r\'esultat permet de g\'en\'eraliser un th\'eor\`eme de Ch. Pittet et L. Saloff-Coste concernant les marches al\'eatoires sur les groupes de type fini minimax et virtuellement r\'esolubles. En outre, l'article identifie des propri\'et\'es conserv\'ees dans le processus de couverture, telles que la classe de r\'esolubilit\'e et la classe de nilpotence du sous-groupe de Fitting. Enfin, nous d\'eterminons exactement quels groupes de type infini minimax et virtuellement r\'esolubles admettent \'egalement ce type de couverture.}

\vspace{20pt}

20P05, 22D05, 60G50, 60B15 (Secondary).

\vspace{10pt}


\end{abstract}

\let\thefootnote\relax\footnote{The authors began work on this paper as participants in the Research in Pairs Program of the {\it Mathematisches Forschungsinstitut Oberwolfach} from March 22 to April 11, 2015. In addition, the project was partially supported by EPSRC Grant EP/N007328/1. Finally, the second author would like to express his gratitude to the {\it Universit\"at Wien} for hosting him for part of the time during which the article was written.}

\tableofcontents

\section{Introduction}

In this paper, we study {\it virtually solvable minimax groups}; these are groups $G$  that possess
a series 

\begin{equation*} 1=G_0\unlhd  G_1\unlhd\cdots \unlhd G_r=G\end{equation*}

\noindent such that each factor $G_i/G_{i-1}$ is either finite, infinite cyclic, or quasicyclic. (Recall that a group is {\it quasicyclic} if, for some prime $p$, it is isomorphic to $\mathbb Z(p^\infty):=\mathbb Z[1/p]/\mathbb Z$.) Denoting the class of virtually solvable minimax groups by $\mathfrak{M}$, we determine which $\mathfrak{M}$-groups can be realized as quotients of virtually torsion-free $\mathfrak{M}$-groups. Moreover, our results on quotients allow us to settle a longstanding question about a lower bound for the return probability of 
a random walk on the Cayley graph of a finitely generated $\mathfrak{M}$-group (see \S 1.3). We are indebted to Lison Jacoboni for pointing out the relevance of our work to this question.  

The importance of $\mathfrak{M}$-groups arises primarily from the special status among virtually solvable groups that is enjoyed by finitely generated $\mathfrak{M}$-groups. As shown by the first author in \cite{kropholler2}, the latter comprise all the finitely generated, virtually solvable groups without any sections isomorphic to a wreath product of a finite cyclic group with an infinite cyclic one. In particular, any finitely generated, virtually solvable group of finite abelian section rank is minimax (a property first established by D. J. S. Robinson \cite{robinson}). For background on these and other properties of $\mathfrak{M}$-groups, we refer the reader to J. C. Lennox and Robinson's treatise \cite{robinson-lennox} on infinite solvable groups.

Within the class $\mathfrak{M}$, we distinguish two subclasses: first, the subclass $\mathfrak{M_1}$ consisting of all the $\mathfrak{M}$-groups that are virtually torsion-free; second, the complement of $\mathfrak{M_1}$ in $\mathfrak{M}$, which we denote $\mathfrak{M_{\infty}}$.  
It has long been apparent that the groups in $\mathfrak{M_1}$ possess a far more transparent  
structure than those in $\mathfrak{M}_{\infty}$. For example, an $\mathfrak{M}$-group belongs to $\mathfrak{M_1}$ if and only if it is residually finite. As a result, every finitely generated $\mathfrak{M_\infty}$-group fails to be linear over any field. In contrast, $\mathfrak{M_1}$-groups are all linear over $\mathbb Q$ and hence can be studied with the aid of the entire arsenal of the theory of linear groups over $\mathbb R$, including via embeddings into Lie groups. The latter approach is particularly fruitful when tackling problems of an analytic nature, such as those that arise in the investigation of random walks (see \cite{pittet}).  

A further consequence of the $\mathbb Q$-linearity of $\mathfrak{M_1}$-groups is that the finitely generated ones fall into merely countably many isomorphism classes. On the other hand, there are uncountably many nonisomorphic, finitely generated $\mathfrak{M_{\infty}}$-groups (see either  [{\bf 16}, p. 104] or Proposition 7.12 below). Other differences between  $\mathfrak{M_1}$-groups and $\mathfrak{M_{\infty}}$-groups are evident in their respective algorithmic properties. The word problem, for instance, is solvable for every finitely generated group in $\mathfrak{M_1}$ (see \cite{cannonito}). However, since the number of possible algorithms is countable, there exist uncountably many nonisomorphic, finitely generated $\mathfrak{M_{\infty}}$-groups with unsolvable word problem.  

Our goal here is to explore the following two questions concerning the relationship between $\mathfrak{M_1}$-groups and $\mathfrak{M_{\infty}}$-groups. In phrasing the questions, we employ a parlance that will be used throughout the paper, and that also occurs in its title: if a group $G$ can be expressed as a homomorphic image of a group $G^\ast$, we say that $G$ is {\it covered by $G^\ast$}. In this case, $G^\ast$ is referred to as a {\it cover} of $G$ and any epimorphism $G^\ast\to G$ as a {\it covering}. 
A cover of $G$ that belongs to the class $\mathfrak{M}_1$ is called an {\it $\mathfrak{M_1}$-cover} of $G$, and the corresponding covering is designated an {\it $\mathfrak{M_1}$-covering}. 

\begin{question1} Under what conditions is it possible to cover an $\mathfrak{M_{\infty}}$-group by an  $\mathfrak{M_1}$-group?
\end{question1}

\begin{question2} If an $\mathfrak{M_{\infty}}$-group $G$ can be covered by an $\mathfrak{M_1}$-group $G^\ast$, 
how can we choose $G^\ast$ so that it retains many of the properties enjoyed by $G$?
\end{question2}

Answering the above questions should enable us to reduce certain problems about $\mathfrak{M}$-groups to the more tractable case where the group belongs to $\mathfrak{M_1}$.  A current problem inviting such an approach arises in Ch. Pittet and L. Saloff-Coste's study \cite{pittet} of random walks on the Cayley graphs of finitely generated $\mathfrak{M}$-groups. Their paper establishes a lower bound on the probability of return for this sort of random walk when the group is virtually torsion-free, but their methods fail to apply to $\mathfrak{M_{\infty}}$-groups.* \footnote{*The hypothesis that the group is virtually torsion-free should be included in the statement of [{\bf 21}, Theorem 1.1], for the proof requires that assumption. We thank Lison Jacoboni for bringing this mistake to our attention.} One way to extend their bound to the latter case is to    
prove that any finitely generated member of the class $\mathfrak{M_{\infty}}$ can be expressed as a homomorphic image of an $\mathfrak{M_1}$-group. In Theorem 1.5 below, we establish that the condition that the group be finitely generated is indeed one possible answer to Question 1.1, thus yielding the desired generalization of Pittet and Saloff-Coste's result (Corollary 1.8).

\subsection{Structure of $\mathfrak{M}$-groups}

Before stating our main results, we summarize the structural properties of $\mathfrak{M}$-groups in Proposition 1.3 below; proofs of these may be found in \cite{robinson-lennox}.  In the statement of the proposition, as well as throughout the rest of the paper, we write ${\rm Fitt}(G)$ for the {\it Fitting subgroup} of a group $G$, namely, the subgroup generated by all the nilpotent normal subgroups.  In addition, we define $R(G)$ to be the {\it finite residual} of $G$, by which we mean the intersection of all the subgroups of finite index.

\begin{proposition} If $G$ is an $\mathfrak{M}$-group, then the following four statements hold.

\begin{enumerate*}

\item ${\rm Fitt}(G)$ is nilpotent and $G/{\rm Fitt}(G)$ virtually abelian. 

\item If $G$ belongs to $\mathfrak{M_1}$, then $G/{\rm Fitt}(G)$ is finitely generated. 

\item $R(G)$ is a direct product of finitely many quasicyclic groups. 
 
\item $G$ is a member of $\mathfrak{M_1}$ if and only if $R(G)=1$. \hfill\(\square\)
\end{enumerate*}

\end{proposition}

\subsection{Covers of finitely generated $\mathfrak{M}$-groups}

Our main result, Theorem 1.13, characterizes all the $\mathfrak{M}$-groups that can be covered by an $\mathfrak{M_1}$-group. We will postpone describing this theorem until \S 1.4, focusing first on 
its implications for finitely generated $\mathfrak{M}$-groups. We begin
with this special case because of its immediate relevance to random walks, as well as its potential to find further applications. 

While discussing our results, we will refer to the following example; it is the simplest instance of a finitely generated $\mathfrak{M_\infty}$-group, originally due to P. Hall \cite{hall}. More sophisticated examples of such groups are described in the final section of the paper. 

\begin{example} {\rm Take $p$ to be a prime, and let $G^\ast$ be the group of upper triangular $3\times 3$ matrices $( a_{ij})$ with entries in the ring $\mathbb Z[1/p]$ such that $a_{11}=a_{33}=1$ and $a_{22}$ is a power of $p$.  Let  $A$ be the central subgroup consisting of all the matrices $( a_{ij})\in G^\ast$ with $a_{22}=1$, $a_{12}=a_{23}=0$, and $a_{13}\in \mathbb Z$. 
Set $G=G^\ast/A$. Then $G$ is a finitely generated solvable minimax group with a quasicyclic center.}
\end{example}

In Example 1.4, there is an epimorphism from the finitely generated, torsion-free solvable minimax group $G^\ast$ to $G$.
The kernel of this epimorphism is a central cyclic subgroup of $G^\ast$. Moreover, the groups $G$ and $G^\ast$ are structurally very similar. For instance, they have the same derived length, and the nilpotency classes of ${\rm Fitt}(G)$ and ${\rm Fitt}(G^\ast)$ coincide. 

It turns out that the group $G$ in Example 1.4 is, in certain respects, quite typical for a finitely generated $\mathfrak{M_{\infty}}$-group. In Theorem 1.5 below, we prove that such groups always admit an $\mathfrak{M_1}$-covering exhibiting most of the properties manifested by the covering in Example 1.4. In the statement of the theorem, the {\it spectrum}
of a group $G$, written ${\rm spec}(G)$, is the set of primes $p$ for which $G$ has a section isomorphic to $\mathbb Z(p^\infty)$. Also, ${\rm solv}(G)$ denotes the {\it solvable radical} of $G$, that is, the subgroup generated by all the solvable normal subgroups. Note that, if $G$ is an $\mathfrak{M}$-group, ${\rm solv}(G)$ is a solvable normal subgroup of finite index. Another piece of notation refers to the derived length of $G$ if $G$ is solvable, written ${\rm der}(G)$. Finally, if $N$ is a nilpotent group, then its nilpotency class is denoted ${\rm nil}\ N$. 

\begin{theorem} Let $G$ be a finitely generated $\mathfrak{M}$-group, and write $N={\rm Fitt}(G)$ and $S={\rm solv}(G)$.   
 Then there is a finitely generated $\mathfrak{M_1}$-group $G^\ast$ and an epimorphism $\phi: G^\ast\to G$ satisfying the following four properties, where $N^\ast={\rm Fitt}(G^\ast)$ and $S^\ast={\rm solv}(G^\ast)$. 
\begin{enumerate*}

\item ${\rm spec}(G^\ast)={\rm spec}(G)$.

\item $N^\ast=\phi^{-1}(N)$; hence $S^\ast=\phi^{-1}(S)$.

\item ${\rm der}(S^\ast)={\rm der}(S)$. 

\item ${\rm nil}\ N^\ast={\rm nil}\ N$ and ${\rm der}(N^\ast)={\rm der}(N)$.

\end{enumerate*}
\end{theorem}

The first part of statement (ii) in Theorem 1.5 implies that ${\rm Ker}\ \phi$ is nilpotent, but the theorem says little more about the kernel's properties. In particular, Theorem 1.5 makes no claims about the kernel of the covering being central or polycyclic, although both properties hold for the kernel in Example 1.4. Indeed, it is not possible in general to form a covering whose kernel has either of these two attributes. We illustrate this phenomenon in \S 7 with two examples, one where the spectrum contains two primes and another where it consists of a single prime.  

The reader will notice that, in all of the examples appearing in the paper, the kernel of the covering is abelian. Nevertheless, as will become clear in the proof of Theorem 1.5, our construction may conceivably produce a kernel with nilpotency class larger than one (see Remark 6.2).  Whether this is merely an artifact of the techniques employed here, or whether certain groups will only admit coverings with nonabelian kernels, remains a mystery.  

\begin{openquestion1} Can the group $G^\ast$ and epimorphism $\phi:G^\ast\to G$ in Theorem 1.5 always be chosen so that ${\rm Ker}\ \phi$ is abelian?
\end{openquestion1}

\noindent We conjecture that this question has a negative answer.

\begin{acknowledgement}{\rm We are grateful to an anonymous referee for posing Open Question 1.6.}
\end{acknowledgement}

\subsection{Application to random walks}

We now describe in detail the application of Theorem 1.5 to random walks alluded to above.  For the statement of the result, we require some notation specific to this matter. First, if $f$ and $g$ are functions from the positive integers to the nonnegative real numbers, we write $f(m)\succsim g(m)$ whenever there are positive constants $a$, $b$, and $c$ such that $f(m)\geq ag(bm)$ for $m>c$. Also, if $f(m)\succsim g(m)$ and $g(m)\succsim f(m)$, then we write $f(m)\sim g(m)$.

Suppose that $G$ is a group with a finite symmetric generating set $S$; by {\it symmetric}, we mean that it is closed under inversion. 
Consider the simple random walk on the Cayley graph of $G$ with respect to $S$. For any positive integer $m$, let $P_{(G,S)}(2m)$ be the probability of returning to one's starting position after $2m$ steps. It is shown in \cite{pittet3} that,
for any other finite symmetric generating set $T$, $P_{(G,T)}(2m)\sim P_{(G,S)}(2m)$. 

In \cite{pittet} the following lower bound is established for the probability of return for a random walk on the Cayley graph of a finitely generated $\mathfrak{M_1}$-group.  

\begin{theorem}{\rm (Pittet and Saloff-Coste [{\bf 21}, Theorem 1.1])} Let $G$ be an $\mathfrak{M_1}$-group with a finite symmetric generating set $S$. Then

\[
\pushQED{\qed}
P_{(G,S)}(2m)\succsim\exp(-m^{\frac{1}{3}}).\qedhere
\popQED
\]

\end{theorem}

With the aid of Theorem 1.5, we generalize Theorem 1.7 to all finitely generated $\mathfrak{M}$-groups. As related at the end of the paper in the acknowledgements, this corollary and its proof were communicated to the authors by Lison Jacoboni. 

\begin{corollary}  Let $G$ be an $\mathfrak{M}$-group with a finite symmetric generating set $S$. Then
$$P_{(G,S)}(2m)\succsim\exp(-m^{\frac{1}{3}}).$$ 
\end{corollary}

\begin{proof} According to Theorem 1.5, there is an $\mathfrak{M_1}$-group $G^\ast$ and an epimorphism $\phi:G^\ast \to G$. Also, we can select $G^\ast$ so that it has a finite symmetric generating set $S^\ast$ with $\phi(S^\ast)=S$.  Consider random walks on the Cayley graphs of both groups with respect to these generating sets, with both walks commencing at the identity element. Then $P_{(G,S)}(2m)$ is equal to the probability that the random walk on $G^\ast$ will take us to an element of ${\rm Ker}\ \phi$ after $2m$ steps. As a consequence, $P_{(G,S)}(2m)\geq P_{(G^\ast,S^\ast)}(2m)$. Moreover, by Theorem 1.7, we have $P_{(G^\ast,S^\ast)}(2m)\succsim\exp(-m^{\frac{1}{3}}).$ Hence $P_{(G,S)}(2m)\succsim\exp(-m^{\frac{1}{3}}).$
\end{proof}

It is proved in \cite{pittet3} that the probability of return for groups of exponential growth is always bounded above by the function $\exp(-m^{\frac{1}{3}})$. Hence we have the

\begin{corollary} Let $G$ be an $\mathfrak{M}$-group with a finite symmetric generating set $S$. If $G$ has exponential growth, then

\[
\pushQED{\qed}
P_{(G,S)}(2m)\sim\exp(-m^{\frac{1}{3}}).\qedhere
\popQED
\]
\end{corollary}

Recall that a finitely generated solvable group has either exponential or polynomial growth, with the latter property holding if and only if the group is virtually nilpotent (see \cite{milnor} and \cite{wolf}). Furthermore, N. Varopoulos shows in \cite{varo} that any finitely generated group of polynomial growth of degree $d$ has probability of return $m^{-d/2}$.

\subsection{Covers of non-finitely-generated $\mathfrak{M}$-groups}

Having discussed finitely generated $\mathfrak{M}$-groups, we turn next to the question
of what other $\mathfrak{M}$-groups can be realized as quotients of $\mathfrak{M_1}$-groups. As we shall see below, not all $\mathfrak{M}$-groups admit such a description; moreover, even if there is an $\mathfrak{M_1}$-cover, it may not be possible to construct one with the same spectrum as the group. To distinguish $\mathfrak{M}$-groups with different spectra, we will employ the symbols $\mathfrak{M^\pi}$ and $\mathfrak{M^\pi_1}$ to denote the subclasses of $\mathfrak{M}$ and $\mathfrak{M_1}$, respectively, 
consisting of all those members with a spectrum contained in a set of primes $\pi$. Alternatively, an $\mathfrak{M^\pi}$-group will be referred to as a {\it $\pi$-minimax group}. In Proposition 1.11 below (proved in \S 3.4), we identify three properties that must be satisfied by an $\mathfrak{M}$-group in order for it to have an $\mathfrak{M_1^\pi}$-cover.
One of these properties requires the following notion, which will play a prominent role in the paper.

\begin{definition}{\rm Let $\pi$ be a set of primes.  A {\it $\pi$-number} is a nonzero integer whose prime divisors all belong to $\pi$. If $G$ is a group and $A$ a $\mathbb ZG$-module, then we say that the action of $G$ on $A$ is {\it $\pi$-integral} if, for each $g\in G$, there are integers $\alpha_0,\alpha_1, \dots, \alpha_m$ such that $\alpha_m$ is a $\pi$-number and $(\alpha_0+\alpha_1g+\cdots+\alpha_mg^m)\in {\rm Ann}_{\mathbb ZG}(A)$. }
\end{definition}

Considering inverses of elements of $G$, we immediately see that we can replace the condition that $\alpha_m$ is a $\pi$-number in Definition 1.10 with the condition that $\alpha_0$ is a $\pi$-number. Indeed, as the reader may easily verify, this is also equivalent to requiring that both $\alpha_0$ and $\alpha_m$ are $\pi$-numbers. 

\begin{proposition} Let $\pi$ be a set of primes and $G$ an $\mathfrak{M}$-group with Fitting subgroup $N$. Set $Q=G/N$. If $G$ is a homomorphic image of an $\mathfrak{M_1^\pi}$-group, then $G$ satisfies the following three conditions. 

\begin{enumerateabc}

\item $Q$ is finitely generated.

\item ${\rm spec}(N)\subseteq \pi$.

\item $Q$ acts $\pi$-integrally on $N_{\rm ab}$. 
\end{enumerateabc}

\end{proposition}
 
We point out that statement (b) in Proposition 1.11 is obvious and (a) is an immediate consequence of Proposition 1.3(ii).
The proof of the third statement is also quite straightforward and is based on the fact that, for any $\mathbb ZG$-module $A$ whose underlying
abelian group is in $\mathfrak{M^\pi_1}$, the action of $Q$ on $A$ is $\pi$-integral (Lemma 3.21).

In Example 1.12, we describe some situations where either statement (a) or (c) in Proposition 1.11 is false while the other two assertions
hold. These examples will serve two purposes: first, to exhibit $\mathfrak{M_{\infty}}$-groups that cannot be expressed 
as a homomorphic image of any $\mathfrak{M_1}$-group; second, to demonstrate that there are $\mathfrak{M_{\infty}}$-groups $G$ that cannot be realized as 
quotients of $\mathfrak{M_1}$-groups with the same spectrum as $G$, but nevertheless occur as quotients of $\mathfrak{M_1}$-groups with larger spectra. The descriptions 
involve the ring $\mathbb Z_p$ of $p$-adic integers and its multiplicative group of units, denoted $\mathbb Z_p^\ast$.

\begin{example} {\rm Let $p$ be a prime. The groups that we discuss are denoted $G_1$, $G_2$, and $G_3$. In each case, $G_i$ is defined to be a semidirect product $\mathbb Z(p^\infty)\rtimes \Lambda_i$, where $\Lambda_i$ is a certain minimax subgroup of $\mathbb Z_p^\ast$. In this semidirect product, we assume that the action of 
$\Lambda_i$ on $\mathbb Z(p^\infty)$ arises from the natural $\mathbb Z_p$-module structure on $\mathbb Z(p^\infty)$. Note that, with this definition, we necessarily have ${\rm Fitt}(G_i)\cong \mathbb Z(p^\infty)$.
\begin{itemize}

\item For $i=1$, we take $q$ to be a prime such that $p\ |\ q-1$. The subgroup of $\mathbb Z_p^\ast$ consisting of all the $p$-adic integers congruent to $1$ modulo $p$ is isomorphic to $\mathbb Z_p$ if $p$ is odd and  $\mathbb Z_2\oplus (\mathbb Z/2)$ if $p=2$. Thus $\mathbb Z^\ast_p$ has a subgroup isomorphic to $\mathbb Z[1/q]$ that contains $q$. Choose $\Lambda_1$ to be such a subgroup. Hence $G_1$ satisfies conditions (b) and (c) in Proposition 1.11 for $\pi=\{p,q\}$. However, $G_1$ fails to fulfill (a) and so cannot be covered by an $\mathfrak{M_1}$-group. (This example also appears in [{\bf 16}, p. 92].) 

\item For $i=2$, let $\Lambda_2$ be a cyclic subgroup generated by an element of $\mathbb Z_p^\ast$ that is 
transcendental over $\mathbb Q$. For the group $G_2$, assertion (a) from Proposition 1.11 is true, and (b) holds for $\pi=\{p\}$. However, condition (c) in Proposition 1.11 is not satisfied for any set of primes $\pi$. Hence there is no $\mathfrak{M_1}$-group that has a quotient isomorphic to $G_2$. 

\item For $i=3$, let $q$ be a prime distinct from $p$, and take $\Lambda_3$ to be the cyclic subgroup of $\mathbb Z^\ast_p$ generated by $q$. 
For the group $G_3$, statement (a) in Proposition 1.11 is true, and (b) holds for $\pi=\{p\}$. Condition (c), on the other hand, is not satisfied for $\pi=\{p\}$. Hence Proposition 1.11 implies that $G_3$ cannot be realized as a quotient of a $p$-minimax group. However, (b) and (c) are both fulfilled for $\pi=\{p,q\}$. In fact, $G_3$ is a quotient of the torsion-free $\{p,q\}$-minimax group $\mathbb Z[1/pq]\rtimes C_{\infty}$, where the generator of $C_{\infty}$ acts on $\mathbb Z[1/pq]$ by multiplication by $q$.
\end{itemize}

}

\end{example}

The main theorem of the paper, Theorem \ref{thm2}, encompasses Proposition 1.11 and its converse, thus providing a complete characterization of all the $\mathfrak{M}$-groups that occur as quotients of $\mathfrak{M_1^\pi}$-groups.  At the same time, the theorem identifies another set of conditions that is equivalent to properties (a), (b), and (c) from Proposition 1.11.

\begin{theorem}\label{thm2}
Let $\pi$ be a set of primes and $G$ an $\mathfrak{M}$-group. Write $N={\rm Fitt}(G)$ and $Q=G/N$. Then the following three statements are equivalent. 
\begin{enumerate!}

\item $G$ can be expressed as a homomorphic image of an $\mathfrak{M_1^\pi}$-group. 

\item $Q$ is finitely generated, ${\rm spec}(N)\subseteq \pi$, and $Q$ acts $\pi$-integrally on $N_{\rm ab}$. 

\item $Q$ is finitely generated, ${\rm spec}(N)\subseteq \pi$, and $Q$ acts $\pi$-integrally on the image of $R(G)$ in $N_{\rm ab}$.  

\end{enumerate!}

Moreover, if the above conditions are fulfilled, then we can select an $\mathfrak{M_1^\pi}$-minimax group $G^\ast$ and epimorphism $\phi:G^\ast\to G$ so that properties {\rm (ii)-(iv)} in Theorem 1.5 hold.

\end{theorem}

\begin{remark} {\rm It follows from Theorem \ref{thm2} that the solvable minimax groups $G$ that are quotients of $\mathfrak{M_1^\pi}$-groups for $\pi={\rm spec}(G)$ are precisely those solvable minimax groups that belong to the class $\mathcal{U}$ from \cite{loren}.  In that paper, the second author shows that the groups in this class enjoy two cohomological properties that are not manifested by all solvable minimax groups. }

\end{remark}

Our next lemma shows that Theorem 1.5  is a special case of the implication (III)$\implies$(I) in Theorem \ref{thm2}.

\begin{lemma} If $G$ is a finitely generated $\mathfrak{M}$-group, then $G$ satisfies statement {\rm (III)} in Theorem \ref{thm2}
for $\pi={\rm spec}(G)$. 

\end{lemma}

\begin{proof} Since $Q$ is finitely presented, $N_{\rm ab}$ must be finitely generated as a $\mathbb ZQ$-module. Because $\mathbb ZQ$ is a Noetherian ring, it follows that $N_{\rm ab}$ is a Noetherian module. Therefore $N_{\rm ab}$ must be virtually torsion-free, implying that the image of $R(G)$ in $N_{\rm ab}$ is trivial. Hence (III) holds. 
\end{proof}

After laying the foundations for our argument in \S \S 3-5, we prove Theorem \ref{thm2} in \S 6.  In that section, we also discuss how the theorem can be extended to yield a $\pi$-minimax cover that is entirely torsion-free, rather than just virtually torsion-free (Corollaries 6.3 and 6.4).  With a torsion-free cover, however, we are unable to achieve quite the same degree of resemblance between the cover and the original group.

\section{Strategy and terminology}

\subsection{Strategy for proving Theorem \ref{thm2}}

In proving Theorem \ref{thm2}, we introduce a new method for studying $\mathfrak{M}$-groups, one that we hope will give rise to further advances in the theory of such groups. Our approach involves embedding an $\mathfrak{M}$-group $G$ densely in a locally compact, totally disconnected topological group and taking advantage of certain features of the structure of this new group. In this section, we describe the main aspects of the argument for (II)$\implies$(I), highlighting the roles played by four pivotal propositions, {\bf Propositions 4.15, 3.12, 4.20, and 3.28}. 
For our discussion here, we will assume that $G$ has torsion at merely a single prime $p$. This means that the finite residual $P$ of $G$ is a direct product of finitely many quasicyclic $p$-groups.

The first step is to densely embed $N$ in a nilpotent, locally compact topological group $N_p$ such that the compact subgroups of $N_p$ are all polycyclic pro-$p$ groups. 
The group $N_p$ is constructed by forming the direct limit of the pro-$p$ completions of the finitely generated subgroups of $N$. In the case where $N$ is abelian and written additively, this is equivalent to tensoring $N$ with $\mathbb Z_p$. The precise definition and properties of $N_p$, called the {\it tensor $p$-completion} of $N$, are discussed in \S 5. Moreover, in \S 4, we investigate the class $\mathfrak{N}_p$ of topological groups to which these tensor $p$-completions belong. 

Drawing on a technique originally due to Peter Hilton, we build a topological group extension $1\rightarrow N_p\rightarrow G_{(N,p)}\rightarrow Q\rightarrow 1$ that fits into a commutative diagram of the form 

\begin{displaymath} \begin{CD}
1 @>>> N @>>> G @>>> Q @>>> 1\\
&& @VVV @VVV @| &&\\  
1 @>>> N_p @>>> G_{(N,p)} @>>> Q @>>> 1,
  \end{CD} \end{displaymath}

\noindent and in which $N_p$ is an open normal subgroup. We now employ the first of our four key propositions, {\bf Proposition 4.15}, which allows us to obtain two closed subgroups $R_0$ and $X$ of $G_{(N,p)}$ satisfying the following properties.

\begin{enumerateabc}

\item $R_0$ is a radicable normal subgroup of $G_{(N,p)}$ contained in $N_p$. (A group $R$ is said to be {\it radicable} if the map $r\mapsto r^m$ is a surjection $R\to R$ for any $m\in \mathbb Z$.)                                       
\item $G_{(N,p)}=R_0X$.  
\item $X\cap N_p$ is a polycyclic pro-$p$-group. 
\end{enumerateabc}

\noindent Notice that, since $Q$ is virtually polycyclic, (c) implies that $X$ must be virtually torsion-free.

At this stage, we will invoke the most important proposition for the proof, {\bf Proposition 3.12}. This result  will enable us to construct a sequence 

\begin{equation} \cdots \stackrel{\psi_{i+1}}{\longrightarrow}R_i\stackrel{\psi_i}{\longrightarrow} R_{i-1}\stackrel{\psi_{i-1}}{\longrightarrow}\cdots \stackrel{\psi_2}{\longrightarrow} R_1\stackrel{\phi_1}{\longrightarrow} R_0\end{equation}
 
\noindent of group epimorphisms such that every $R_i$ is a radicable nilpotent group upon which $X$ acts as a group of automorphisms and each $\psi_i$ commutes with the action of $X$.  In addition, each $R_i$ will contain an isomorphic copy of $P$, and the preimage under $\psi_i$ of the copy of $P$ in $R_{i-1}$ will be the copy of $P$ in $R_i$. Finally, the map $P\to P$ induced by $\psi_i$ will be the epimorphism $u\mapsto u^p$.

In the case that $R_0$ is abelian, such a sequence (2.1) can be formed very easily merely by making each $R_i=R_0$ and $\psi_i$ the $p$-powering map.
Extending this construction to nilpotent groups, however, constitutes one of the most subtle aspects of the paper. Our method is based on the observation that $H^2(R,A)\cong H^2(P,A)^R$ whenever $R$ is a radicable nilpotent group with torsion subgroup $P$ and $A$ is a finite $\mathbb ZR$-module (Lemma 3.13). Gleaned from the Lyndon-Hochschild-Serre (LHS) spectral sequence, this property results from the fact that most of the cohomology of a torsion-free radicable nilpotent group vanishes when the coefficient module is finite. It is precisely this feature of radicable nilpotent groups that makes passing into $N_p$ desirable.

Our next step is to consider the inverse limit $\Omega$ of (2.1), letting $\psi$ be the canonical epimorphism $\Omega\to R_0$. The preimage $\Psi$ of $P$ in $\Omega$ is the inverse limit of the sequence $\cdots \rightarrow P\rightarrow P\rightarrow  P$, where each homomorphism $P\to P$ is the $p$-power map. This means that $\Psi$ is isomorphic to the direct sum of finitely many copies of $\mathbb Q_p$, so that $\Omega$ is torsion-free. We define a topology on $\Omega$ that makes it a locally compact, totally disconnected group and the action of $X$ on $\Omega$ continuous.  
Thus the semidirect product $\Gamma:=\Omega\rtimes X$ is a virtually torsion-free, locally compact group that covers $G_{(N,p)}=R_0X$ via the continuous epimorphism $(r,x)\mapsto \psi(r)x$, where $r\in \Omega$ and $x\in X$. Moreover, it will turn out that, like $G$, the group $\Gamma$ acts $\pi$-integrally on the abelianization of its Fitting subgroup. 

Inside the group $\Gamma$, we assemble a $\pi$-minimax cover for $G$, making use of {\bf Propositions 4.20 and 3.28}. The former permits us to find a $\pi$-minimax nilpotent cover for $N$ within $\Gamma$. From this cover of $N$, we stitch together a cover for $G$; to show that the latter cover is $\pi$-minimax, we exploit the $\pi$-integrality of the action of $\Gamma$ on the abelianization of its Fitting subgroup, applying {\bf Proposition 3.28.}

\subsection{Notation and terminology}

{\it Rings and modules}
\vspace{5pt}

Let $p$ be a prime. Then $\mathbb Z_p$ is the ring of $p$-adic integers, $\mathbb Z_p^\ast$ is the multiplicative group of units in $\mathbb Z_p$, $\mathbb Q_p$ is the field
 of $p$-adic rational numbers, and $\mathbb F_p$ is the field with $p$ elements. 

Let $\pi$ be a set of primes. A {\it $\pi$-number} is a nonzero integer whose prime divisors all belong to $\pi$. The ring $\mathbb Z[\pi^{-1}]$ is the subring of $\mathbb Q$ consisting of all rational numbers of the form $m/n$ where $m\in \mathbb Z$ and $n$ is a $\pi$-number. 

The term {\it module} will always refer to a left module. Moreover, if $R$ is a ring and $A$ an $R$-module, we will write the operation $R\times A\to A$ as $(r,  a)\mapsto r\cdot a$. 

A {\it section} of a module is a quotient of a submodule. 
 
If $G$ is a group and $R$ a ring, then $RG$ denotes the group ring of $G$ over $R$. 
\vspace{8pt}

{\it Abstract groups}
\vspace{5pt}

If $p$ is a prime, then $\mathbb Z(p^\infty)$ is the {\it quasicyclic $p$-group}, that is, the inductive limit of the cyclic $p$-groups $\mathbb Z/p^k$ for $k\geq 1$.

Let $G$ be a group. The center of $G$ is written $Z(G)$. The {\it Fitting subgroup} of $G$, denoted ${\rm Fitt}(G)$, is the subgroup generated by all the nilpotent normal subgroups of $G$. We use $R(G)$ to represent the {\it finite residual} of $G$, which is the intersection of all the subgroups of finite index.  

A {\it section} of a group is a quotient of a subgroup. 

A {\it series} in a group $G$ will always mean a series of the form 

$$1=G_0\unlhd G_1\unlhd \cdots \unlhd G_r=G.$$

If $\phi:G\to Q$ and $\psi:H\to Q$ are group homomorphisms, then $$G\times_QH:=\{(g,h)\ |\ g\in G, h\in H, {\rm and}\ \phi(g)=\psi(h)\}.$$

If $g$ and $h$ are elements of a group $G$, then $g^h:=h^{-1}gh$ and $[g,h]:=g^{-1}g^h$. If $H$ and $K$ are subgroups of $G$, then $[H, K]:=\langle [h,k]\ | \ h\in H, k\in K\rangle$. Moreover, if $H_1,\dots, H_r$ are subgroups of $G$, then 
$[H_1,\dots, H_r]$ is defined recursively as follows: $[H_1,\dots, H_r]=[[H_1,\dots, H_{r-1}], H_r]$. For $H, K\leq G$ and $i\geq 1$, we abbreviate $[H, \underbrace{K,\dots, K}_{i}]$ to $[H,\ _iK]$. 

Let $G$ be a group. The derived series of $G$ is written

$$\cdots \leq G^{(1)}\leq G^{(0)}=G.$$

\noindent  The subgroups $G^{(1)}$ and $G^{(2)}$ will often be represented by $G'$ and $G''$, respectively. Also, the lower and upper central series are written

$$\cdots \leq \gamma_2G\leq \gamma_1G=G\ \ \ {\rm and}\ \ \ 1=Z_0(G)\leq Z_1(G)\leq \cdots,$$

\noindent respectively.

Let $\pi$ be a set of primes. A {\it $\pi$-torsion} group, also referred to as a {\it $\pi$-group}, is one in which the order of every element is a $\pi$-number. 
A group $G$ is said to be {\it $\pi$-radicable} if, for any $\pi$-number $n$ and $g\in G$, there is an $x\in G$ such that $x^n=g$. A group is {\it radicable} if it is $\pi$-radicable for $\pi$ the set of all primes. Throughout the paper, we will often write abelian groups additively, in which case we will use the term {\it divisible} rather than {\it radicable}. 

A {\it \v{C}ernikov group} is a group that is a finite extension of a direct product of finitely many quasicyclic groups.

A virtually solvable group $G$ is said to have {\it finite torsion-free rank} if it possesses a subgroup series of finite length whose factors are either infinite cyclic or torsion. Because of the Schreier refinement theorem, the number of infinite cyclic factors in any such a series is an invariant of $G$, called its {\it Hirsch length} and denoted $h(G)$.  

A virtually solvable group is said to have {\it finite abelian section rank} if its elementary abelian sections are all finite.

If $G$ is an $\mathfrak{M}$-group and $p$ a prime, then $m_p(G)$ represents the number of factors isomorphic to $\mathbb Z(p^\infty)$ in any series in which each factor is finite, cyclic, or quasicyclic. By Schreier's refinement theorem, this number does not depend on the particular series selected.

The {\it spectrum} of an $\mathfrak{M}$-group $G$, denoted ${\rm spec}(G)$, is the set of primes $p$ such that $m_p(G)~>~0$. Note that this set is necessarily finite. 

Let $G$ and $K$ be groups. If $G$ acts upon $K$ on the left, then we call $K$ a {\it $G$-group}. In this case, we write the operation $G\times K\to K$ as $(g,  k)\mapsto g\cdot k$. A homomorphism between $G$-groups that commutes with the action of $G$ is a {\it $G$-group homomorphism}.

\vspace{8pt}

{\it Topological groups}
\vspace{5pt}

If $G$ is an abstract group, then the {\it profinite completion} of $G$, denoted $\hat{G}$, is the inverse limit of all the finite quotients of $G$.
We regard $\hat{G}$ as a topological group with respect to the topology on this inverse limit induced by the product topology. This topology makes $\hat{G}$ both compact and totally disconnected. The canonical homomorphism from $G$ to $\hat{G}$ is called the {\it profinite completion map} and denoted $c^G:G\to \hat{G}$. 

If $G$ is an abstract group and $p$ a prime, then the {\it pro-$p$ completion} of $G$, denoted $\hat{G}_p$, is the inverse limit of all the  quotients of $G$ that are finite $p$-groups. As with the profinite completion, $\hat{G}_p$ is a compact, totally disconnected topological group. The canonical homomorphism from $G$ to $\hat{G}_p$ is called the {\it pro-$p$ completion map} and denoted $c_p^G:G\to \hat{G}_p$. 

Let $G$ be a topological group. If $H\leq G$, then $\overline{H}$ denotes the closure of $H$ in $G$.  We say that $G$ is {\it topologically finitely generated} if there are elements $g_1,\dots, g_r$ of $G$ such that $G=\overline{\langle g_1,\dots, g_r\rangle}$.  If $G$ happens to be generated as an abstract group by finitely many elements, then we say that $G$ is {\it abstractly finitely generated}. 

Let $p$ be a prime. A pro-$p$ group $G$ is called {\it cyclic} if there is an element $g$ of $G$ such that $G=\overline{\langle g\rangle}$. Hence any cyclic pro-$p$ group must be either a finite cyclic $p$-group or a copy of $\mathbb Z_p$. 
 A {\it polycyclic pro-$p$ group} is a pro-$p$ group with a series of finite length in which each factor is a cyclic pro-$p$ group. 
Since we will also be applying the adjectives {\it cyclic} and {\it polycyclic} to abstract groups, we will adhere to the convention that, when used in their pro-$p$ senses, the terms
{\it cyclic} and {\it polycyclic} will always be immediately followed by the adjective {\it pro-$p$}. 

A topological space is {\it $\sigma$-compact} if it is the union of countably many compact subspaces. 

 A topological group is {\it locally elliptic} if every compact subset is contained in a compact open subgroup.

\begin{acknowledgement} {\rm We thank an anonymous referee for pointing out the relevance of local ellipticity
to our discussion in \S 4.}
\end{acknowledgement}

\section{Preliminary results about abstract groups}

In this section, we establish an array of results concerning abstract groups that we require for the proof of Theorem \ref{thm2}.
Foremost among these is Proposition 3.12, which is proved in \S 3.2. For the parts of Theorem \ref{thm2} addressing nilpotency class and derived length, we 
need to investigate ${\rm nil}\ G$ and ${\rm der}(G)$ for a group $G$ that can be written as a product of two subgroups, one of which is normal. 
This is accomplished in \S 3.1, whose principal result is Proposition 3.8.  In \S 3.3, we discuss coverings of groups, and finally, 
in \S 3.4, we study $\pi$-integral actions.

\subsection{Nilpotent and solvable actions}

In order to calculate the nilpotency class of a product, we require the well-known concept of a nilpotent action.

\begin{definition} {\rm Let $G$ be a group and $K$ a $G$-group. We
define the {\it lower $G$-central
series}
\[ \dots \leq \gamma^G_3K \leq   \gamma^G_2K \leq
\gamma^G_1K\] of $K$ as follows: $\gamma^G_1K=K$;
$\gamma^G_iK=\langle k(g\cdot l)k^{-1}l^{-1}\ | \ k\in K, l\in
\gamma^G_{i-1}K, g\in G\rangle$ for
$i>1$.

We say that the action
of
$G$ on
$K$ is {\it nilpotent} if
there is a nonnegative integer
$c$ such that $\gamma^G_{c+1}N=1$. The smallest
such integer $c$ is called the {\it nilpotency class} of the action, written ${\rm nil}_G\ N$.} 
\end{definition}

Dual to the lower $G$-central series is the upper $G$-central series.

\begin{definition} {\rm Let $G$ be a group and $K$ a $G$-group. Define $Z^G(K)$ to be the subgroup of $Z(K)$ consisting of all the elements that are centralized by $G$. The {\it upper $G$-central series} 

$$Z^G_0(K)\leq Z^G_1(K)\leq Z^G_2(K)\leq \cdots$$

\noindent is defined as follows: $Z^G_0(K)=1$; $Z^G_i(K)/Z^G_{i-1}(K)=Z^G(K/Z_{i-1}^G(K))$ for $i\geq 1$. }
\end{definition}

As stated in the following well-known lemma, the lengths of the upper and lower $G$-central series are the same.

\begin{lemma} Let $K$ be a $G$-group. If $i\geq 0$, then $\gamma^G_{i+1}K=1$ if and only if  $Z^G_i(K)=K$. \nolinebreak \hfill\(\square\)
\end{lemma}

Analogous to the notion of a nilpotent action is that of a solvable action, defined below. Although we are unaware of any specific reference, it is likely that this concept has been studied before.

\begin{definition} {\rm Let $G$ be a group and $K$ a $G$-group. We
define the {\it
$G$-derived series}
\[ \dots \leq \delta^G_2K \leq   \delta^G_1K \leq
\delta^G_0K\] of $K$ as follows: $\delta^G_0K=K$; $\delta^G_iK=\gamma^{G^{(i-1)}}_2(\delta^G_{i-1}K)$ for $i\geq 1$. 

 The action
of
$G$ on
$K$ is {\it solvable} if
there is a nonnegative integer
$d$ such that $\delta^G_{d}K=1$. The smallest
such integer $d$ is called the {\it derived length} of the action, denoted ${\rm der}_G(K)$. }
\end{definition}

We wish to use the lower $G$-central series and the $G$-derived series to describe the lower central series and derived series of a group that can be written as a product of two subgroups, one of which is normal. 
First, however, we mention an elementary property of commutator subgroups, Lemma 3.5. This can be proved
with the aid of Hall and L. Kalu\v{z}nin's ``three subgroup lemma" [{\bf 16}, 1.2.3]; we leave the details to the reader. 

\begin{lemma} Let $G$ be a group and $K$ a normal subgroup of $G$. For each $i>0$, 

\[
\pushQED{\qed}
[K, \gamma_i G]\leq  [K,\ _i G].\qedhere
\popQED
\]
\end{lemma}

In the next lemma, we investigate the lower central series and derived series of a product.

\begin{lemma} Let $G$ be a group such that $G=KX$, where $K\unlhd G$ and $X\leq G$.  For each $i>0$, 

$$\gamma_i G=[K,\ _{i-1}G]\gamma_i X\ \ {\rm and}\ \  G^{(i)}=[K, G, G',\dots, G^{(i-1)}]X^{(i)}.$$
\end{lemma}

\begin{proof} We content ourselves with deriving the first equation; the second may be deduced by a similar argument. We proceed by induction on $i$. For $i=1$, the equation clearly holds. Suppose $i>1$. Let $g\in \gamma_{i-1}G$ and $h\in G$. By the inductive hypothesis, $g=ax$, where $a\in [K,\ _{i-2} G]$ and $x\in \gamma_{i-1}X$. Also, write $h=by$ with $b\in K$ and $y\in X$. We have

\begin{equation*} [g,h]=[ax,by]=[a,y]^x\ [x,y]\ [a,b]^{xy}\ [x,b]^y.\end{equation*}

\noindent In this product, the first and third factors plainly belong to $[K,\ _{i-1} G]$, and the second to $\gamma_i X$. The fourth, meanwhile, is an element of $[K, \gamma_{i-1} G]$. But, by Lemma 3.5, this latter subgroup
is contained in  $[K,\ _{i-1} G]$.  As a result, $[g,h]\in [K,\ _{i-1}G]\gamma_i X.$ We have thus shown $\gamma_i G=[K,\ _{i-1}G]\gamma_i X$.
\end{proof}

The first factors in the representations of $\gamma_i G$ and $G^{(i)}$ provided in Lemma 3.6 permit an alternative description, namely, as terms in the lower $X$-central series 
and $X$-derived series of $K$. 

\begin{lemma} Let $G$ be a group such that $G=KX$, where $K\unlhd G$ and $X\leq G$. Then, for all $i\geq 1$,

$$\gamma^X_i K=[K,\ _{i-1} G]\ \ \mbox{\rm and} \ \ \delta^X_i K=[K, G, G',\dots, G^{(i-1)}].$$ 
\end{lemma}

\begin{proof} The first equation follows immediately from the definition of $\gamma^X_i K$. 
For the second equation, we proceed by induction on $i$. We have $\delta^X_1 K=\gamma^X_2 K$, and, by the first equation, $\gamma^X_2 K=[K,G]$. This establishes
the case $i=1$. Assume next that $i>1$. Then 

$$\delta^X_i K=\gamma^{X^{(i-1)}}_2 (\delta^X_{i-1} K)= \gamma^{X^{(i-1)}}_2([K,G,G',\dots, G^{(i-2)}]).$$

\noindent Moreover, from the first equation, we obtain

$$\gamma^{X^{(i-1)}}_2([K,G,G',\dots, G^{(i-2)}])=[K,G,G',\dots, G^{(i-2)},H],$$

\noindent where $H=[K,G,G',\dots, G^{(i-2)}]X^{(i-1)}$. But, by Lemma 3.6, $H=G^{(i-1)}$, yielding the second equation.
\end{proof}

Combining the preceding two lemmas gives rise to

\begin{proposition} Let $G$ be a group such that $G=KX$, where $K\unlhd G$ and $X\leq G$. Then statements {\rm (i)} and {\rm (ii)} below hold.
\begin{enumerate*}

\item $G$ is nilpotent if and only if $X$ is nilpotent and acts nilpotently on K. In this case,
$${\rm nil}\ G={\rm max}  \{ {\rm nil}\ X, {\rm nil}_X\, K\}.$$

\item $G$ is solvable if and only if $X$ and $K$ are solvable. In this case,

\[
\pushQED{\qed}
{\rm der}(G)={\rm max}  \{{\rm der}(X), {\rm der}_X(K)\}.\qedhere
\popQED
\]
\end{enumerate*}

\end{proposition}

\subsection{Radicable nilpotent groups}

In this subsection, we examine the attributes of radicable nilpotent groups that underpin the proof of Theorem 1.13. We start with three very basic, and undoubtedly well-known, properties.

\begin{lemma} The three statements below hold for any nilpotent group $N$. 
\begin{enumerate*}

\item $N$ is radicable if and only if $N_{\rm ab}$ is divisible.

\item If $N$ is radicable, then $\gamma_i N$ is radicable for all $i\geq 1$.

\item If $K\unlhd N$ such that $K$ and $N/K$ are radicable, then $N$ is radicable.
\end{enumerate*}
\end{lemma}

\begin{proof} 
Assertions (i) and (ii) can be deduced using the epimorphisms from the tensor powers of $N_{\rm ab}$ to the factors in the lower central series of $N$ induced by the iterated commutator maps. In addition, one requires the obvious fact that the property of radicability is preserved by central extensions. 

To prove (iii), we set $Q=N/K$ and consider the exact sequence $K_{\rm ab}\rightarrow N_{\rm ab}\rightarrow Q_{\rm ab}\rightarrow 0$ of abelian groups. Since $K_{\rm ab}$ and $Q_{\rm ab}$ are divisible, so is $N_{\rm ab}$.
It follows, then, from (i) that $N$ is radicable.
\end{proof}

Next we show that the terms in the lower central $G$-series and derived $G$-series of a nilpotent radicable $G$-group are radicable.

\begin{lemma} Let $G$ be a group and $R$ a radicable nilpotent $G$-group. Then, for all $i\geq 1$, $\gamma_i^G R$ and $\delta^G_i R$ are radicable. 
\end{lemma}

\begin{proof}  Both conclusions will follow immediately from Lemma 3.7 if we prove that, whenever $R$ is a radicable normal nilpotent subgroup of a group $G$, the subgroup $[R,G]$ must be radicable. 
To show this, observe $[R,G]/R'= IA$, where $A$ is the $\mathbb ZG$-module $R_{\rm ab}$ and $I$ is the augmentation ideal in $\mathbb ZG$. Since $A$ is divisible as an abelian group, $IA$ is too. Moreover, Lemma 3.9(ii) implies that $R'$ must be radicable. Therefore, by Lemma 3.9(iii), $[R,G]$ is radicable. 
\end{proof}

Lemma 3.10 gives rise to the following two observations about extensions. 

\begin{lemma} Let $G$ be a group and $1\rightarrow F\rightarrow \bar{R}\rightarrow R\rightarrow 1$ an extension of $G$-groups such that $F$ is finite and $\bar{R}$ is nilpotent and radicable.
\begin{enumerate*}

\item If $G$ acts nilpotently on $R$, then $G$ also acts nilpotently on $\bar{R}$ and ${\rm nil}_G\ \bar{R}={\rm nil}_G\ R$. 
\item If $G$ acts solvably on $R$, then $G$ also acts solvably on $\bar{R}$ and ${\rm der}_G(\bar{R})={\rm der}_G(R)$. 
\end{enumerate*}

\end{lemma}

\begin{proof} We just prove (i), the proof of (ii) being similar. Let $c={\rm nil}_G\ R$. Then $\gamma_{c+1}^G \bar{R}$ is finite. But $\gamma_{c+1}^G \bar{R}$ is radicable by Lemma 3.10. Therefore  $\gamma_{c+1}^G \bar{R}=1$. Hence ${\rm nil}_G\ \bar{R}=c.$ 
\end{proof}

The most important property of radicable nilpotent groups for our purposes is described in Proposition 3.12 below. 

\begin{proposition} Let $X$ be a group that acts upon a nilpotent radicable group $R$.  Assume further that the torsion subgroup $P$ of $R$ is $p$-minimax for some prime $p$. Then 
there exist a nilpotent radicable $X$-group $\bar{R}$ and an $X$-group epimorphism $\psi:\bar{R}\to R$ 
such that there is an $X$-group isomorphism $\nu: P\to \psi^{-1}(P)$ with $\psi\nu(u)=u^p$ for all $u\in P$.
Furthermore, it follows that
${\rm Ann}_{\mathbb ZX}(\bar{R}_{\rm ab})= {\rm Ann}_{\mathbb ZX}(R_{\rm ab}).$

\end{proposition}

We prove Proposition 3.12 using the cohomological classification of group extensions. Our argument hinges on the following lemma concerning the cohomology of radicable nilpotent groups.

\begin{lemma} Let $R$ be a radicable nilpotent group with torsion subgroup $P$. For any finite $\mathbb ZR$-module $A$,  $H^2(R,A)$ is mapped isomorphically onto $H^2(P,A)^R$ by the restriction homomorphism.
\end{lemma} 

Lemma 3.13 will be proved using the following fact derived from the LHS spectral sequence, which is a special case of  
[{\bf 12}, Theorem 2].

\begin{lemma} Let $1\rightarrow K\rightarrow G\rightarrow Q\rightarrow 1$ be a group extension and $A$ a $\mathbb ZG$-module. Suppose 
that the cohomology groups $H^1(K,A)$, $H^2(Q,A^K)$, and $H^3(Q,A^K)$ are all trivial. Then $H^2(G,A)$ is mapped isomorphically onto $H^2(K,A)^G$ by the restriction homomorphism. \nolinebreak \hfill\(\square\)
\end{lemma}

\begin{proof}[Proof of Lemma 3.13] Notice first that $R$ must act trivially on $A$. Hence, by Lemma 3.14, the conclusion will follow if we show $H^n(R/P,A)=0$ for $n=2, 3$. To verify this, we employ the universal coefficient exact sequence
\begin{equation} 0\longrightarrow {\rm Ext}^1_{\mathbb Z}(H_{n-1}(R/P,\mathbb Z),A)\longrightarrow H^n(R/P,A)\longrightarrow {\rm Hom}_{\mathbb Z}(H_n(R/P,\mathbb Z),A)\longrightarrow 0 \end{equation}

\noindent for $n\geq 1$. Since $R/P$ is torsion-free, nilpotent, and radicable, $H_n(R/P, \mathbb Z)$ is torsion-free and divisible for $n\geq 1$ (see, for instance, [{\bf 10}, Proposition 4.8]). As a result, the second and fourth groups in sequence (3.1) are both trivial. Thus $H^n(R/P,A)=0$ for $n\geq 1$. 
\end{proof}

In proving Proposition 3.12, we  will also avail ourselves of the following simple cohomological property
of radicable abelian groups.

\begin{lemma} Let $p$ be a prime and $R$ a radicable abelian group whose $p$-torsion subgroup is minimax. In addition, 
define $\rho: R\to R$ by $\rho(r)=r^p$ for every $r\in R$. Finally, set $A={\rm Ker}\ \rho$. Then the cohomology class of the group extension $1\rightarrow A\rightarrow R\stackrel{\rho}{\rightarrow} R\rightarrow 1$ is fixed by the canonical action of ${\rm Aut}(R)$ on
$H^2(R,A)$.

\end{lemma}

\begin{proof} Let $\xi \in H^2(R,A)$ be the cohomology class of the extension $1\rightarrow A\rightarrow R\stackrel{\rho}{\rightarrow} R\rightarrow 1$. Take $\alpha\in {\rm Aut}(R)$. Then the diagram 

\begin{displaymath} \begin{CD}
1 @>>> A @>>> R @>\rho>> R @>>> 1\\
&& @VV\alpha V @VV\alpha V @VV\alpha V &&\\
1 @>>> A @>>> R @>\rho>> R @>>> 1  
\end{CD} \end{displaymath}

\noindent commutes. As a result, we have  $\alpha\cdot \xi=\xi$.
\end{proof}

Finally, we require the group-theoretic lemma below. 

\begin{lemma} Let $G$ be a group with a finite normal subgroup $F$.  Let $\phi: R\to G$ and $\psi:R\to G$ be homomorphisms where $R$ is a radicable group.  If $\phi(r)F=\psi(r)F$ for all $r\in R$, then $\phi=\psi$. 
\end{lemma}

\begin{proof} Since $\psi(R)$ is radicable, it must centralize $F$. As a result, the map $r\mapsto \phi(r)\psi(r^{-1})$ defines a homomorphism $R\to F$.  But all such homomorphisms are trivial, so that  $\phi=\psi$. 
\end{proof}

Armed with the preceding lemmas, we embark on the proof of Proposition 3.12. 

\begin{proof}[Proof of Proposition 3.12] Take $A$ to be the subgroup of $P$ consisting of all its elements of order dividing $p$. 
The group $\bar{R}$ will be obtained by employing cohomology with coefficients in the $\mathbb ZR$-module $A$. By Lemma 3.13, we know  that the restriction map induces a natural isomorphism $\theta: H^2(R,A)\to H^2(P,A)^R$.  Consider now the short exact sequence $1\rightarrow A\rightarrow P\stackrel{\rho}{\rightarrow} P\rightarrow 1,$ where $\rho(u)=u^p$ for all $u\in P$.  Let $\xi \in H^2(P,A)$ be the cohomology class of this extension. Then Lemma 3.15 implies $\xi\in H^2(P,A)^R$.  Thus there is a unique element $\zeta$ of $H^2(R,A)$ such that $\theta(\zeta)=\xi$. Regarding $H^2(P,A)$ as a $\mathbb ZX$-module and invoking Lemma 3.15 again, we have $x\cdot \xi=\xi$ for all $x\in X$. It follows, then, from the naturality of $\theta$ that $x\cdot \zeta=\zeta$ for all $x\in X$. 

Form a group extension $1\rightarrow A\stackrel{\iota}{\rightarrow}\bar{R}\stackrel{\psi}{\rightarrow} R\rightarrow 1$ corresponding to $\zeta\in H^2(R,A)$. Then $\bar{R}$ is plainly nilpotent. Moreover, because $\theta(\zeta)=\xi$, there is an isomorphism $\nu: P\to \psi^{-1}(P)$ such that the diagram

\begin{displaymath} \begin{CD}
1 @>>> A @>>> P @>\rho>> P @>>> 1\\
&& @| @VV\nu V  @| &&\\
1 @>>> A @>\iota>> \psi^{-1}(P) @>\psi>> P @>>> 1   
\end{CD} \end{displaymath}

\noindent commutes. 
Thus $\psi\nu(u)=u^p$ for all $u\in P$. Also, we see that $\bar{R}$ is an extension of $P$ by $R/P$. Hence $\bar{R}$ is radicable by Lemma 3.9(iii). 

For each $x\in X$, let $\alpha_x$ and $\beta_x$ be the automorphisms of $A$ and $R$, 
respectively, that are induced by $x$. Since $x\cdot \zeta=\zeta$ for all $x\in X$, we can find, for each $x\in X$, an automorphism $\gamma_x: \bar{R}\to \bar{R}$ that renders the diagram  

\begin{displaymath} \begin{CD}
1 @>>> A @>\iota>> \bar{R} @>\psi>> R @>>> 1\\
&& @VV\alpha_xV @VV\gamma_x V @VV\beta_x V&&\\
1 @>>> A @>\iota>> \bar{R} @>\psi>> R @>>> 1   
\end{CD} \end{displaymath}

\noindent commutative. Furthermore, in view of Lemma 3.16, $\gamma_x$ must be the only automorphism of $\bar{R}$ that makes this diagram commute. As a result, the assignment $x\mapsto \gamma_x$ defines a homomorphism
from $X$ to ${\rm Aut}(\bar{R})$, thus equipping $\bar{R}$ with an action of $X$. This action plainly makes $\psi$ into an $X$-group homomorphism. We claim that the same holds for $\nu$.  To see this, notice that $\nu$ induces an $X$-group homomorphism $P/A\to \psi^{-1}(P)/\iota(A)$. Hence it is easy to deduce from Lemma 3.16 that $\nu$ is an $X$-group homomorphism.   

To verify the assertion about the annihilators, let $\lambda\in {\rm Ann}_{\mathbb ZX}(R_{\rm ab})$.  Consider now the $\mathbb ZX$-module epimorphism  $\bar{R}_{\rm ab}\rightarrow R_{\rm ab}$ induced by $\psi$. In light of Lemma 3.16, the fact that $\lambda$ annihilates $R_{\rm ab}$ implies that $\lambda$ must have the same effect on $\bar{R}_{\rm ab}$. 
\end{proof}

\subsection{Elementary properties of coverings}

In this subsection, we collect several fundamental facts concerning coverings that are required for the proofs of our main results. We begin
by making some elementary observations about extending a covering of a quotient to the entire group; the proofs of these are left to the reader.  

\begin{lemma} Let $1\rightarrow K\stackrel{\iota}{\rightarrow}G\stackrel{\epsilon}{\rightarrow} Q\rightarrow 1$ be a group extension. Suppose that there is a group $Q^\ast$ equipped with an epimorphism $\psi:Q^\ast\to Q$. Let $G^\ast=G\times_QQ^\ast$ and define the homomorphisms $\phi:G^\ast\to G$, $\iota^\ast:K\to G^\ast$, and $\epsilon^\ast:G^\ast\to Q^\ast$ as follows:

$\phi(g,q)=g$ for all $(g,q)\in G^\ast$; 

$\iota^\ast(k)=(\iota(k),1)$ for all $k\in K$;

$\epsilon^\ast (g,q)=q$ for all $(g,q)\in G^\ast$. 

\noindent Then the diagram

\begin{displaymath} \begin{CD}
1 @>>> K @>\iota^\ast>> G^\ast @>\epsilon^\ast>> Q^\ast @>>> 1\\
&& @| @VV\phi V @VV\psi V &&\\
1 @>>> K @>\iota>> G @>\epsilon>>  Q @>>> 1  
\end{CD} \end{displaymath}

\noindent  commutes. In addition, the following three statements hold.
\begin{enumerate*}

\item ${\rm solv}(G^\ast)={\rm solv}(G)\times_Q{\rm solv}(Q^\ast)$.

\item ${\rm Fitt}(G^\ast)={\rm Fitt}(G)\times_Q{\rm Fitt}(Q^\ast)$.

\item If $N$ is a nilpotent subgroup of $G$ of class $c$ such that $\psi^{-1}(\epsilon(N))$ is nilpotent of class $\leq c$, then $\phi^{-1}(N)$ is nilpotent of class $c$. \hfill\(\square\)
\end{enumerate*} 

\end{lemma}

We now briefly touch on coverings of finite groups.

\begin{lemma} Let $G$ be a finite group with a normal nilpotent subgroup $N$ such that ${\rm nil}\ N=c$ and ${\rm der}(N)=d$. Then there is a torsion-free, virtually polycyclic group $G^\ast$ and an epimorphism $\phi:G^\ast\to G$ such that 

$${\rm nil}\ \phi^{-1}(N)={\rm max}\{c, 1\}\ {\rm and}\ {\rm der}(\phi^{-1}(N))= {\rm max}\{d, 1\}.$$
\end{lemma}

\begin{proof} Put $c'={\rm max}\{c, 1\}$ and $d'={\rm max}\{d, 1\}$. Let $F$ be a finitely generated free group admitting an epimorphism $\theta: F\to G$.  Define $K=\phi^{-1}(N)$ and $L=K^{(d')}\gamma_{c'+1} K$. Notice $L\leq {\rm Ker}\ \theta$.  Moreover, according to [{\bf 26}, Theorem 3], $K/L $ is torsion-free.  Also, $F/K'$ is torsion-free by virtue of [{\bf 9}, Theorem 2]. It follows, then, that $F/L$ is torsion-free. Hence $F/L$
is a cover of $G$ that fulfills our requirements.
\end{proof}

We finish our preliminary discussion of coverings by proving a lemma describing a well-known way to cover a finite $\mathbb ZG$-module, followed by a corollary stating that every $\mathbb ZG$-module in $\mathfrak{M_1^\pi}$
has a $\mathbb ZG$-module cover that is torsion-free and $\pi$-minimax.

\begin{lemma} Let $G$ be a group and $A$ a finite $\mathbb ZG$-module. Then there is a $\mathbb ZG$-module $A^\ast$ and an $\mathbb ZG$-module epimorphism
$\phi:A^\ast\to A$ such that $A^\ast$ is torsion-free and polycyclic as an abelian group.
\end{lemma}

\begin{proof} Let $A^\ast$ be the free abelian group on the set of elements of $A$, and let $\phi^\ast:A^\ast\to A$ be the group epimorphism that maps
each of these generators to itself. There is, then, an obvious $\mathbb ZG$-module structure on $A^\ast$ that makes
$\phi$ into an $\mathbb ZG$-module epimorphism. 
\end{proof}

\begin{corollary} Let $G$ be a group and $A$ a $\mathbb ZG$-module whose underlying additive group is virtually torsion-free and $\pi$-minimax. Then there is a $\mathbb ZG$-module $A^\ast$ and a $\mathbb ZG$-module epimorphism
$\phi:A^\ast\to A$ such that $A^\ast$ is torsion-free and $\pi$-minimax as an abelian group.
\end{corollary}

\begin{proof} Let $A_0$ be a $\mathbb ZG$-submodule of $A$ such that $A/A_0$ is finite and $A_0$ is torsion-free {\it qua} abelian group. According to Lemma 3.19, $A/A_0$ can be covered by a $\mathbb ZG$-module $B$ whose additive group is torsion-free and polycyclic. Hence letting $A^\ast=A\times_{A/A_0} B$ establishes the conclusion.
\end{proof}

\subsection{$\pi$-Integral actions }

We conclude \S 3 by investigating $\pi$-integral actions (see Definition 1.10), focusing on their importance for identifying $\pi$-minimax groups (Propositions 3.25 and 3.28).  The most elementary aspect of the connection between $\pi$-integrality and the $\pi$-minimax property is captured in our first lemma. 

\begin{lemma} Let $\pi$ be a set of primes and $G$ a group. If $A$ is a $\mathbb ZG$-module that is virtually torsion-free and $\pi$-minimax as an abelian group, then $G$ acts $\pi$-integrally on $A$.  
\end{lemma}

\begin{proof} In view of Corollary 3.20, we can suppose that $A$ is torsion-free as an abelian group. Take $\alpha:A\to A$ to be an automorphism of $A$ {\it qua} abelian group arising from the action of an element of $G$. This automorphism induces an automorphism $\alpha_{\pi}$ of $A\otimes \mathbb Z[\pi^{-1}]$. Moreover,  $\alpha_{\pi}$ is a root of a monic polynomial with coefficients in $\mathbb Z[\pi^{-1}]$. Multiplying by a large enough integer yields a polynomial $f(t)\in \mathbb Z[t]$ such that $f(\alpha_{\pi})=0$ and the leading coefficient of $f(t)$ is a $\pi$-number. It follows, then, that $f(\alpha)=0$. 
Therefore $G$ acts $\pi$-integrally on $A$. 
\end{proof}

Lemma 3.21 allows us to establish Proposition 1.11.

\begin{proof}[Proof of Proposition 1.11]

Statement (b) is plainly true since every quotient of a $\pi$-minimax group is $\pi$-minimax. To show (a) and (c), let $\phi: G^\ast\to G$ be an epimorphism where $G^\ast$ is an $\mathfrak{M_1^\pi}$-group. Put $N^\ast={\rm Fitt}(G^\ast)$ and $c={\rm nil}\ N^\ast$. By Proposition 1.3(ii), $G^\ast/N^\ast$ is finitely generated. Since $\phi(N^\ast)\leq N$, it follows that statement (a) holds. 

Now we establish assertion (c).  That $N^\ast$ is virtually torsion-free implies that $Z_i(N^\ast)/Z_{i-1}(N^\ast)$ is virtually torsion-free for $i\geq 1$. According to Lemma 3.21, this means that, for $i=1,\dots, c$, $G$ acts $\pi$-integrally on $Z_i(N^\ast)/Z_{i-1}(N^\ast)$. Taking the images of the subgroups $Z_i(N^\ast)$ under the composition $N^\ast\stackrel{\phi}{\to} N\to N_{\rm ab}$ yields a chain
$0=A_0\subseteq A_1\subseteq\cdots \subseteq A_c$ of $\mathbb ZG$-submodules of $N_{\rm ab}$. 
Because $Q$ acts $\pi$-integrally on each factor in this chain, the action of $Q$ on $A_c$ must be $\pi$-integral. Moreover, $N_{\rm ab}/A_c$ is polycyclic since $G^\ast/N^\ast$ is polycyclic. Hence $G$ acts $\emptyset$-integrally on $N_{\rm ab}/A_c$ and therefore $\pi$-integrally on $N_{\rm ab}$. 
\end{proof}

Next we prove that the $\pi$-integrality property is inherited by tensor products.

\begin{lemma} Let $\pi$ be a set of primes and $G$ a group. If $G$ acts $\pi$-integrally on the $\mathbb ZG$-modules $A$ and $B$, then $G$ acts $\pi$-integrally on $A\otimes B$, viewed as a $\mathbb ZG$-module via the diagonal action.
\end{lemma}

The above lemma follows readily from

\begin{lemma}
Let $A$ and $B$ be abelian groups and $\pi$ a set of primes. Let $\phi\in {\rm End}(A)$ and $\psi\in {\rm End}(B)$. Suppose that there are polynomials $f(t), g(t)\in \mathbb Z[t]$ such that $f(\phi)=0$, $g(\psi)=0$, and the leading coefficients of $f(t)$ and $g(t)$ are $\pi$-numbers. Then there is a polynomial $F(t)\in \mathbb Z[t]$ such that the leading coefficient of $F(t)$ is a $\pi$-number and $F(\phi\otimes \psi)=0$. 
\end{lemma}

\begin{proof} 

Let $\theta$ be the ring homomorphism from the polynomial ring $\mathbb Z[s,t]$ to the ring ${\rm End}(A\otimes B)$ such that $\theta(s)=\phi\otimes \mathbb 1_B$ and $\theta(t)=\mathbb 1_A\otimes \psi$. 
Define $I$ to be the ideal of $\mathbb Z[s,t]$ generated by $f(s)$ and $g(t)$. Then $\theta(I)=0$. Put $R=\mathbb Z[\pi^{-1}]$, and take $J$ to be the ideal in the polynomial ring $R[s,t]$ generated by $f(s)$ and $g(t)$. The images of $s$ and $t$ in $R[s,t]/J$ are both integral over $R$. Hence the same is true for the image of $st$ in $R[s,t]/J$. In other words, there is a monic polynomial $F'$ over $R$ such that $F'(st)=p(s,t)f(s)+q(s,t)g(t)$, where $p(s,t)$ and $q(s,t)$ are polynomials in $R[s,t]$. 
Multiplying by a large enough $\pi$-number, we acquire a polynomial $F$ over $\mathbb Z$ whose leading coefficient is a $\pi$-number and such that $F(st)\in I$. It follows, then, that $F(\phi\otimes \psi)=0$. 
\end{proof}

Lemma 3.22 gives rise to the properties below.

\begin{lemma} Let  $\pi$ be a set of primes and $G$ a group. Let $N$ be a $G$-group such that $G$ acts $\pi$-integrally on $N_{\rm ab}$.  Then the following two statements are true.

\begin{enumerate*}
\item The group $G$ acts $\pi$-integrally on $\gamma_iN/\gamma_{i+1}N$ for $i\geq 1$. 
\item If $N$ is nilpotent, then,  
for every $G$-subgroup $M$ of $N$, $G$ acts $\pi$-integrally on $M_{\rm ab}$. 
\end{enumerate*}
\end{lemma}

\begin{proof} 
Assertion (i) follows from Lemma 3.22 and the fact that the iterated commutator map induces a $\mathbb ZG$-module epimorphism from the $i$th tensor power of $N_{\rm ab}$ to $\gamma_iN/\gamma_{i+1}N$.

To prove (ii), set $c={\rm nil}\ N$. Then the $\mathbb ZG$-module $M_{\rm ab}$ has a series $0=M_{c+1}\subseteq M_c\subseteq \cdots \subseteq M_1=M_{\rm ab}$ of submodules such that, for $1\leq i\leq c$, 
$M_i/M_{i+1}$ is isomorphic to a $\mathbb ZG$-module section of $\gamma_iN/\gamma_{i+1}N$. It follows, then, from (i) that $G$ acts $\pi$-integrally on each factor $M_i/M_{i+1}$. Therefore 
$G$ acts $\pi$-integrally on $M_{\rm ab}$. 
\end{proof}

In the next proposition, we examine a group-theoretic application of $\pi$-integrality.  

\begin{proposition} Let $\pi$ be a set of primes and $G$ a group with a normal subgroup $N$ such that the following conditions are satisfied.

{\rm (i)} $N$ is nilpotent and virtually torsion-free.

{\rm (ii)} $G/N$ is virtually polycyclic. 

{\rm (iii)} $G$ acts $\pi$-integrally on $N_{\rm ab}$.

{\rm (iv)} $N$ is generated as a $G$-group by a $\pi$-minimax subgroup. 

\noindent Then $G$ must be $\pi$-minimax. 
\end{proposition}

The first step in proving Proposition 3.25 is the following lemma.

\begin{lemma} Let $\pi$ be a set of primes and $G$ a polycyclic group. Let $A$ be a $\mathbb ZG$-module that is generated as a $\mathbb ZG$-module by a $\pi$-minimax additive subgroup. Assume further that the action of $G$ on $A$ is $\pi$-integral. Then there is a finite subset $\pi_0$ of $\pi$ such that the underlying additive group of $A$ is an extension of a $\pi_0$-torsion group by a $\pi$-minimax one.  \nolinebreak \hfill\(\square\)
\end{lemma}

\begin{proof} Our strategy is to induct on the length $r$ of a series $1=G_0\lhd G_1\lhd \cdots \lhd G_r=G$ in which each factor $G_i/G_{i-1}$ is cyclic. If $r=0$, then $A$ is plainly $\pi$-minimax.   
Suppose $r>0$, and let $B$ be a $\pi$-minimax subgroup of $A$ that generates $A$ as a $\mathbb ZG$-module. Also, take $g\in G$ such that $gG_{r-1}$ is a generator of $G/G_{r-1}$. In addition, let $C$ be the $\mathbb ZG_{r-1}$-submodule of $A$ generated by $B$. Notice that, in view of the inductive hypothesis, the underlying abelian group of $C$ is ($\pi_1$-torsion)-by-($\pi$-minimax) for some finite subset $\pi_1$ of $\pi$. 

 Let $\alpha_0, \alpha_1,\dots , \alpha_m\in \mathbb Z$ such that
$\alpha_0$ and $\alpha_m$ are $\pi$-numbers and \begin{equation} (\alpha_0+\alpha_1g+\cdots +\alpha_mg^m)\cdot a=0\end{equation} for all $a\in A$. Furthermore, let $\pi_0$ be the union of the following three sets: ${\rm spec}(C)$, $\pi_1$, and the set of prime divisors of $\alpha_0\alpha_m$. Then $\pi_0$ is a finite subset of $\pi$. Next set $U=A\otimes \mathbb Z[\pi_0^{-1}]$ and $V=C\otimes \mathbb Z[\pi_0^{-1}]$. Observe that the kernel of the canonical homomorphism $A\to U$ is a $\pi_0$-torsion group, and that $V$ is $\pi$-minimax. 
Since   
$U=\displaystyle{\sum_{i=-\infty}^{\infty} g^i\cdot V}$, equation (3.2) implies $\displaystyle{U= \sum_{i=0}^{m-1} g^i\cdot V}$. Therefore $U$ is $\pi$-minimax, yielding the conclusion of the lemma. 
\end{proof}

\begin{proof}[Proof of Proposition 3.25] Let $G_0$ be a subgroup of finite index in $G$ such that $N\leq G_0$ and $G_0/N$ is polycyclic. Then $N$ is generated as a $G_0$-group by finitely many $\pi$-minimax subgroups. Therefore $N_{\rm ab}$ is generated as a $\mathbb ZG_0$-module by a single $\pi$-minimax subgroup. Hence Lemma 3.26 implies that there is a finite subset $\pi_0$ of $\pi$ such that $N_{\rm ab}$ is an extension of a $\pi_0$-torsion group by one that is $\pi$-minimax. As a result, any tensor power of $N_{\rm ab}$ is also ($\pi_0$-torsion)-by-($\pi$-minimax). The canonical epimorphism from the $i$th tensor power of $N_{\rm ab}$ to $\gamma_iN/\gamma_{i+1}N$, then, yields that $\gamma_iN/\gamma_{i+1}N$ is an extension of the same form. 
Therefore, appealing to Lemma 3.27 below, we can argue by induction on $j$ that $\gamma_{c-j}N$ is $\pi$-minimax for all $j\geq 0$, where $c={\rm nil}\ N$. In particular, $N$ is $\pi$-minimax, and so $G$ is $\pi$-minimax.  
\end{proof}

It remains to prove the elementary lemma below, which will also find later use.

\begin{lemma} Let $\pi$ be a finite set of primes. Let $N$ be a virtually torsion-free nilpotent group. Suppose further that $N$ contains a normal $\pi$-minimax subgroup $M$ such that $N/M$ is $\pi$-torsion. Then $N$ is $\pi$-minimax.
\end{lemma}

\begin{proof} We induct on ${\rm nil}\ N$. Assume  ${\rm nil}\ N=1$. Then $N\otimes \mathbb Z[\pi^{-1}]\cong M\otimes \mathbb Z[\pi^{-1}]$, so that $N\otimes \mathbb Z[\pi^{-1}]$ is $\pi$-minimax. Hence $N$ is $\pi$-minimax. Suppose ${\rm nil}\ N>1$, and let $Z=Z(N)$. Then $N/Z$ and $Z$ are $\pi$-minimax by virtue of the inductive hypothesis. Thus $N$ is $\pi$-minimax. 
\end{proof}

The following consequence of Proposition 3.25 will be invoked in the proof of Theorem \ref{thm2}.

\begin{proposition}Let $\pi$ be a set of primes and $G$ a group with a normal subgroup $N$ such that the following conditions are satisfied.

{\rm (i)} $N$ is nilpotent and virtually torsion-free.

{\rm (ii)} $G/N$ is virtually polycyclic. 

{\rm (iii)} $G$ acts $\pi$-integrally on $N_{\rm ab}$.

\noindent If $H$ is a $\pi$-minimax subgroup of $N$ and $V$ is a finitely generated subgroup of $G$, then 
$\langle H, V\rangle$ is $\pi$-minimax. 
\end{proposition}

For the proof of Proposition 3.28, we need the following simple observation.

\begin{lemma} Let $\pi$ be a set of primes and $N$ a nilpotent group. If $H$ and $K$ are $\pi$-minimax subgroups of $N$, then $\langle H, K\rangle$ is $\pi$-minimax. 
\end{lemma}

\begin{proof} Set $L=\langle H, K\rangle$. Then $L_{\rm ab}$ is generated by the images of $H$ and $K$.
Thus $L_{\rm ab}$ is $\pi$-minimax, which implies that $L$ is $\pi$-minimax.
\end{proof}

\begin{proof}[Proof of Proposition 3.28]
Put $U=\langle H, V\rangle$. We will describe $U$ in a form that permits us to apply Proposition 3.25. First we observe that, as a virtually polycyclic group, the quotient $V/(N\cap V)$ is finitely presented. Hence $N\cap V$ is generated as a $V$-group by a finite set $S$. Take $L$ to be the subgroup of $N$ generated by the set $H\cup S$. Then Lemma 3.29 implies that $L$ is $\pi$-minimax. 
Next define $K$ to be the $V$-subgroup of $N$ generated, as a $V$-group, by $L$. Then $U=KV$.
Hence, because $N\cap V\leq K$, the quotient $U/K$ is virtually polycyclic.  In addition, Lemma 3.24(ii) yields that $U$ acts $\pi$-integrally on $K_{\rm ab}$. Consequently,  $U$ is seen to be $\pi$-minimax by invoking Proposition 3.25.  
\end{proof}

\section{Preliminary results about topological groups}

As explained in \S 2.1, the proof of Theorem 1.13 will utilize a dense embedding of a nilpotent $\mathfrak{M}$-group in a topological group belonging to the class $\mathfrak{N}_p$, defined below.  In the present section, we establish the properties of groups in this class required for the proof of Theorem 1.13. The culmination of this discussion are two factorization results for topological groups with a normal $\mathfrak{N}_p$-subgroup (Propositions 4.15 and 4.20). 

\subsection{Introducing the class $\mathfrak{N}_p$}

\begin{definition} {\rm For any prime $p$, we define $\mathfrak{N}_p$ to be the class of all nilpotent topological groups $N$ that have a series 

\begin{equation} 1=N_0\unlhd  N_1\unlhd\cdots \unlhd N_r=N\end{equation}

\noindent such that, for every $i\geq 1$, $N_i/N_{i-1}$ is of one of the following two types with respect to the quotient topology: 

(i) a cyclic pro-$p$ group; or 

(ii) a discrete quasicyclic $p$-group. 

}
 
\end{definition}

Notice that the compact groups in $\mathfrak{N}_p$ are exactly the polycyclic pro-$p$ groups. Also, every discrete $\mathfrak{N}_p$-group is a \v{C}ernikov $p$-group (see Proposition 1.3(iii)). 
Some more properties of $\mathfrak{N}_p$-groups are described in Lemma 4.2.

\begin{lemma} Let $p$ be a prime and $N$ a topological group in $\mathfrak{N}_p$. Then the following eight statements are true.
\begin{enumerate*}

\item $N$ is locally compact, totally disconnected, $\sigma$-compact, and locally elliptic.

\item  Every open subgroup of $N$ belongs to $\mathfrak{N}_p$.

\item If $U$ is an open normal subgroup of $N$, then $N/U$ belongs to $\mathfrak{N}_p$. 

\item Every compact subgroup of $N$ belongs to $\mathfrak{N}_p$.  

\item Every closed subgroup of $N$ belongs to $\mathfrak{N}_p$.

\item If $M$ is a closed normal subgroup of $N$, then $N/M$ belongs to $\mathfrak{N}_p$.

\item If $G$ is a totally disconnected, locally compact group and $\phi:N\to G$ is a continuous homomorphism, then ${\rm Im}\ \phi$ is closed in $G$. 

\item If $H$ and $K$ are closed subgroups of $N$ with $K\unlhd N$, then $HK$ is also closed. 
\end{enumerate*}
\end{lemma}

\begin{proof}

To show (i), observe that cyclic pro-$p$ groups and discrete quasicyclic $p$-groups enjoy these four properties. 
Moreover, these properties are preserved by extensions of locally compact groups; for the first three, see  [{\bf 28}, Theorem 6.15], and for the last, see  [{\bf 22}, Theorem 2] or [{\bf 4}, Proposition 4.D.6(2)]. It follows, then, that (i) is true. 

Statements (ii), (iii), and (iv) plainly hold if $N$ is a cyclic pro-$p$ group or a discrete quasicyclic $p$-group. The general cases can then be proved by inducting on the length of the series (4.1) and employing the fact that images under quotient maps of open (compact) sets are open (compact).

To prove (v) and (vi), we fix a compact open subgroup $C$ of $N$, and, exploiting the subnormality of $C$, choose a subgroup series 

$$C=M_0\unlhd M_1\unlhd\cdots \unlhd M_s=N.$$ 

\noindent Because $C$ is open, so are the other subgroups in this series. Thus, by (ii), each $M_i$ is a member of $\mathfrak{N}_p$. Also, (iii) implies that $M_i/M_{i-1}$ is a \v{C}ernikov $p$-group for $0\leq i\leq s$. Taking $H$ to be a closed subgroup of $N$, consider the series

$$H\cap C=(H\cap M_0)\unlhd (H\cap M_1)\unlhd\cdots \unlhd (H\cap M_s)=H,$$ 

\noindent in which every subgroup is open in $H$ and each quotient $(H\cap M_i)/(H\cap M_{i-1})$ is a \v{C}ernikov $p$-group. Since $H\cap C$ is compact, it is a polycyclic pro-$p$ group. It follows, then, that $H$ belongs to $\mathfrak{N}_p$. This proves (v).   
 
For (vi), let $\epsilon:N\to N/M$ be the quotient map. The class of polycyclic pro-$p$ groups is closed under forming continuous Hausdorff homomorphic images. Thus $\epsilon(C)$ is a polycyclic pro-$p$ group. Furthermore, $\epsilon(M_i)$ is an open subgroup of $N/M$ for $0\leq i\leq s$. In addition, each quotient $\epsilon(M_i)/\epsilon(M_{i-1})$ is a \v{C}ernikov $p$-group. Therefore $N/M$ is a member of $\mathfrak{N}_p$. 

To prove (vii), we induct on the length of the series (4.1). Since $G$ is Hausdorff, the statement is true for $r=0$.  Suppose $r>0$, and let $L=N_{r-1}$. 
Then $\phi(L)$ is closed in $G$. Let $A=\overline{\phi(N)}/\phi(L)$, and let $\psi: N/L\to A$ be the homomorphism induced by $\phi$.  Our goal is to prove that ${\rm Im}\ \psi$ is closed in $A$, which will imply the desired conclusion. If $N/L$ is compact, this is immediate. Assume, then, that $N/L$ is not compact, which means that it is quasicyclic.  We will argue now that $A$ must be discrete. Suppose that this is not true. Then $A$ has an infinite compact open subgroup $B$. Since ${\rm Im}\ \psi$ is dense in $A$, the subgroup $B\cap {\rm Im}\ \psi$ must be infinite. Consequently, $\psi^{-1}(B)$ is infinite and hence equal to $N/L$. But this is impossible because $B$ is profinite and there are no nontrivial homomorphisms from a quasicyclic group to a profinite group. Therefore $A$ must be discrete, so that ${\rm Im}\ \psi$ is closed in $A$.  

Lastly, we establish assertion (viii). According to (v), $H$ belongs to $\mathfrak{N}_p$. Hence (vii) implies that $HK/K$ is closed in $N/K$.
Therefore $HK$ is closed in $N$.
\end{proof}

\begin{acknowledgement}{\rm Property (vii) in Lemma 4.2 and its proof were provided to the authors by an anonymous referee.}
\end{acknowledgement}

Lemma 4.2(viii) leads to a version of the Schreier refinement theorem for $\mathfrak{N}_p$-groups.

\begin{lemma}Let $p$ be a prime. In an $\mathfrak{N}_p$-group, any two series of closed subgroups have isomorphic refinements that also consist of closed subgroups. 
\end{lemma}

\begin{proof} We can prove the lemma with the same reasoning employed in \cite{robinson2} to establish Schreier's famous result. The argument there relies on the ``Zassenhaus lemma" to construct the refinements. For our purposes, all that we need to verify is that the subgroups obtained in this manner are closed. This property, however, is ensured by Lemma 4.2(viii).
\end{proof}

Lemma 4.3 has the following consequence. 

\begin{corollary} Let $p$ be a prime, and let $N$ be a topological group in the class $\mathfrak{N}_p$.  In any series in $N$ whose factors are all cyclic pro-$p$ groups or quasicyclic $p$-groups, the number of infinite pro-$p$ factors and the number of infinite factors are both invariants. 
\end{corollary}

\begin{proof} Forming a refinement of such a series with closed subgroups fails to affect the number of factors that are isomorphic to $\mathbb Z_p$, as well as the number of infinite factors.  Hence the conclusion follows immediately from the above lemma.
\end{proof}
 
Corollary 4.4 permits us to make the following definition.

\begin{definition} {\rm Let $N$ be a group in $\mathfrak{N}_p$. The {\it $p$-Hirsch length} of $N$, denoted $h_p(N)$, is the number of factors isomorphic to $\mathbb Z_p$  in any series of finite length in which each factor is either a cyclic pro-$p$ group or a quasicyclic group.}

\end{definition} 

In our next lemma, we examine subgroups of finite index in $\mathfrak{N}_p$-groups.

\begin{lemma} Let $p$ be a prime and $N$ an $\mathfrak{N}_p$-group. Then the following four statements hold. 
\begin{enumerate*}

\item Every subgroup of finite index in $N$ is open.  

\item For each natural number $m$, $N$ possesses only finitely many subgroups of index $m$.  

\item $R(N)$ is closed and $N/R(N)$ compact.

\item $R(N)$ is radicable.

\end{enumerate*}

\end{lemma}

\begin{proof} It is straightforward to see that the property that every subgroup of finite index is open is preserved by extensions of locally compact groups. Thus, since this property holds for any cyclic pro-$p$ group and any quasicyclic $p$-group, assertion (i) must be true. 

Next we jump ahead to prove assertion (iii).  First we observe that (i) implies that the profinite completion map $c^N:N\to \hat{N}$ is continuous. Hence (iii) will follow if we show that $c^N$ is surjective. 
Since the profinite completion functor is right-exact, the property of having a surjective profinite completion map is preserved by extensions. Moreover, this property holds for both cyclic pro-$p$ groups and quasicyclic $p$-groups.  Therefore  $c^N $ is surjective. 

Now we prove (ii). Since $c^N:N\to \hat{N}$ is surjective and continuous, $\hat{N}$ must be a polycyclic pro-$p$ group and thus finitely generated as a topological group. As a consequence, the set ${\rm Hom}(N,F)$ is finite for any finite group $F$. Statement (ii), then, follows. 

We establish assertion (iv) by demonstrating that $R:=R(N)$ fails to have any proper subgroups of finite index. This will then imply that $R_{\rm ab}$ must be a divisible abelian group, which will allow us to conclude from Lemma 3.9(i) that $R$ is radicable. Let $H$ be a subgroup of $R$ with finite index.  According to statement (ii), $H$ has only finitely many conjugates in $N$. If we form the intersection of these conjugates, we obtain a normal subgroup $K$ of $N$ such that $K\leq H$ and $[R:K]<\infty$. This means that $N/K$ is a polycyclic pro-$p$ group. Hence $R\leq K$, so that $R=H$. Therefore $R$ has no proper subgroups of finite index. 
\end{proof}

\subsection{Abelian $\mathfrak{N}_p$-groups}

We now focus our attention on the abelian groups in $\mathfrak{N}_p$. First we prove that they are topological $\mathbb Z_p$-modules. 

\begin{lemma} Let $p$ be a prime and $A$ an abelian group in $\mathfrak{N}_p$.

\begin{enumerate*}

\item There is a unique topological $\mathbb Z_p$-module structure on $A$ that extends its $\mathbb Z$-module structure. 

\item  A subgroup $B$ of $A$ is closed if and only if it is a $\mathbb Z_p$-submodule. 

\item If $B$ is an abelian $\mathfrak{N}_p$-group, then a map $\phi:A\to B$ is a continuous homomorphism if and only if it is a $\mathbb Z_p$-module 
homomorphism. 
\end{enumerate*}

\end{lemma}

\begin{proof} Being $\sigma$-compact and locally elliptic, $A$ is isomorphic as a topological group to the inductive limit of its compact open subgroups. If $C$ is a compact open subgroup of $A$, then $C$ is a pro-$p$ group. Hence there is a unique continuous map $\mathbb Z_p\times C\to C$ that renders $C$ a $\mathbb Z_p$-module and restricts to the map $(m,c)\mapsto mc$ from $\mathbb Z\times C$ to $C$.  As a result, there is a unique continuous map $\mathbb Z_p\times A\to A$ that imparts a $\mathbb Z_p$-module structure to $A$ and restricts to the integer-multiplication map $\mathbb Z\times A\to A$. This proves statement (i). Viewing $A$ as an inductive limit in this fashion also permits us to deduce assertions (ii) and (iii) from the compact cases. 
\end{proof}

Lemmas 4.6 and 4.7 allow us to completely describe the structure of abelian $\mathfrak{N}_p$-groups.

\begin{proposition} For any prime $p$, a topological abelian group $A$ belongs to the class $\mathfrak{N}_p$ if and only if $A$ is a direct sum of finitely many groups of the following four types:

{\rm (i)} a finite cyclic $p$-group;

{\rm (ii)} a quasicyclic $p$-group;

{\rm (iii)} a copy of $\mathbb Z_p$;

{\rm (iv)} a copy of $\mathbb Q_p$.

\end{proposition}

\begin{proof} Let $D=R(A)$ and $C=A/D$. Since $D$ is an injective $\mathbb Z_p$-module, we have $A\cong D\oplus C$ as $\mathbb Z_p$-modules. Invoking the classification of injective modules over a principal ideal domain, we observe that $D$ can be expressed as a direct sum of copies of $\mathbb Q_p$ and $\mathbb Z(p^\infty)$. Moreover, since $D$ belongs to $\mathfrak{N}_p$, we can ascertain from Corollary 4.4 that this direct sum can only involve finitely many summands. Also, according to Lemma 4.6(iii), $C$ is a finitely generated $\mathbb Z_p$-module, which means that it can be expressed as a direct sum of finitely many modules of types (i) and (iii). This completes the proof of the proposition. 
\end{proof}

\begin{remark} {\rm Proposition 4.8 also follows from [{\bf 3}, Lemma 7.11].}
\end{remark}

Our understanding of abelian $\mathfrak{N}_p$-groups allows us to examine the lower central series of an  $\mathfrak{N}_p$-group.

\begin{lemma} Let $p$ be a prime. If $N$ is a member of the class $\mathfrak{N}_p$, then the two properties below hold.
\begin{enumerate*} 

\item $\gamma_iN$ is a closed subgroup of $N$ for every $i\geq 1$.

\item The iterated commutator map induces a $\mathbb Z_p$-module epimorphism

$$ \theta_i: \underbrace{N_{\rm ab}\otimes_ {\mathbb Z_p}\cdots \otimes_{\mathbb Z_p} N_{\rm ab}}_{i}\rightarrow \gamma_iN/\gamma_{i+1}N$$

\noindent for each $i\geq 1$.
\end{enumerate*}
\end{lemma}

\begin{proof} We will dispose of (i) and (ii) with a single argument. If $i\geq 0$, forming iterated commutators of weight $i$ defines a continuous function

\begin{equation*}f_i: \underbrace{N/\overline{N'}\times\dots \times N/\overline{N'}}_{i}\rightarrow \overline{\gamma_iN}/\overline{\gamma_{i+1}N}\end{equation*}
that is $\mathbb Z$-linear in each component.  As $f_i$ is thus also $\mathbb Z_p$-linear in each component, it gives rise to a $\mathbb Z_p$-module homomorphism

\begin{displaymath} \theta_i:\underbrace{N/\overline{N'}\otimes_ {\mathbb Z_p}\dots \otimes_{\mathbb Z_p} N/\overline{N'}}_{i}\rightarrow \overline{\gamma_iN}/\overline{\gamma_{i+1}N}.\end{displaymath}

\noindent By Lemma 4.7(ii), ${\rm Im}\ \theta_i$ is closed in $\overline{\gamma_iN}/\overline{\gamma_{i+1} N}$. Therefore, since  ${\rm Im}\ \theta_i=\gamma_i N/\overline{\gamma_{i+1} N}$, the subgroup $\gamma_i N$ must be closed in $N$.
\end{proof}

\subsection{Extensions of finite $p$-groups by $\mathfrak{N}_p$-groups}

For the proof of Theorem 1.13, we will  require the closure property of $\mathfrak{N}_p$ described in our next lemma.

\begin{lemma} Let $p$ be a prime and $N$ a topological group in $\mathfrak{N}_p$. Also, let $F$ be a finite $p$-group. If $1\rightarrow F\rightarrow E\stackrel{\epsilon}{\rightarrow} N\rightarrow 1$ is an abstract group extension with $E$ nilpotent, then there is a unique topology on $E$ that makes it into a member of $\mathfrak{N}_p$ and $\epsilon: E\to N$ a quotient map.
\end{lemma}

\begin{proof} In view of Lemma 4.2(i), there is a family $\{H_i:i\in \mathbb N\}$ of compact open subgroups of $N$ whose union is $N$ such that, for any $i,\ j\in \mathbb N$, there is a $k\in \mathbb N$ with $H_i\cup H_j\subseteq H_k$. For each $i$, we endow $\epsilon^{-1}(H_i)$ with its pro-$p$ topology. 
According to Lemma 4.6(i), every subgroup of finite index in $H_i$ is open. Thus the map $\epsilon^{-1}(H_i)\to H_i$ induced by $\epsilon$ is a quotient map. Now we construct a topology on $E$ by defining $U\subseteq E$ to be open if and only if $U\cap \epsilon^{-1}(H_i)$ is open in $\epsilon^{-1}(H_i)$ for all $i\in \mathbb N$. 
If $i,\ j\in \mathbb N$, then $H_i\cap H_j$ is open in $H_j$ and thus has finite index in $H_j$. Thus $\epsilon^{-1}(H_i)\cap \epsilon^{-1}(H_j)$ has finite index in $\epsilon^{-1}(H_j)$, and so  $\epsilon^{-1}(H_i)\cap \epsilon^{-1}(H_j)$ is open in $\epsilon^{-1}(H_j)$. As this holds for every $j\in \mathbb N$,  $\epsilon^{-1}(H_i)$ is open in $E$ for all $i\in \mathbb N$. Since $E$ is the union of the $\epsilon^{-1}(H_i)$,  it follows from Lemma 4.12 below that $E$ is a topological
group with respect to the topology we have defined. In addition, $\epsilon$ is a quotient map, implying that $F$ is closed.

According to Lemma 4.13 below, $\epsilon^{-1}(H_i)$ is residually finite; that is, it is Hausdorff with respect to its pro-$p$ topology.  Thus $E$ must be Hausdorff, which means that $E$ induces the discrete topology on $F$. As a result, a series like (4.1) in  $N$ extends to a series of the same type
in $E$, and so $E$ belongs to $\mathfrak{N}_p$. This, then, establishes the existence of a topology on $E$ with the desired properties.

To verify the uniqueness of the topology described above, suppose that $\mathcal{T}$ is a topology on $E$ making $E$ into a Hausdorff topological group and $\epsilon: E\to N$ a quotient map. Then, for any $i\in \mathbb N$, $\mathcal{T}$ induces a topology on $\epsilon^{-1}(H_i)$ that renders it a polycyclic pro-$p$ group. Thus the topology induced
on $\epsilon^{-1}(H_i)$ by $\mathcal{T}$ will coincide with its full pro-$p$ topology. This means that $\mathcal{T}$ is identical with the topology defined above. 
\end{proof}

To complete the above argument, it remains to prove the following two lemmas, the first of which will be invoked again in the next section.

\begin{lemma} Let $G$ be an abstract group that is also a topological space. Assume that $G$ is a union of a family of subgroups $\mathcal{H}$ that are all open in the topology on $G$ and such that, for any pair $H, K\in \mathcal{H}$, there is a subgroup $L\in \mathcal{H}$ with $H, K\leq L$. Suppose further that each subgroup in $\mathcal{H}$ is a topological group with respect to this topology.
Then $G$ is a topological group. 
\end{lemma}

\begin{proof} To show that the multiplication map $G\times G\to G$ is continuous, let $x\in G$ and $U_x$ be an open neighborhood of $x$. Also, take $g, h\in G$ such that $gh=x$. Then we can find a subgroup $H\in \mathcal{H}$ such that $g, h, x\in H$.  Since $H$ is a topological group, there are open subsets $U_g$ and $U_h$ of $H$ such that $g\in U_g$, $h\in U_h$, and $U_gU_h\subseteq U_x$. Also, because $H$ is open in $G$, so are $U_g$ and $U_h$.
This establishes the continuity of the multiplication map $G\times G\to G$. Moreover, the continuity of the inversion map $G\to G$ may be deduced by a similar argument. It follows, then, that $G$ is a topological group. 
\end{proof}

\begin{lemma} Let $1\rightarrow F\rightarrow G\rightarrow Q\rightarrow 1$ be an abstract group extension such that $F$ is finite and $Q$ is a polycyclic pro-$p$ group. Then $G$ is residually finite. 
\end{lemma}

\begin{proof} In this argument, we employ the notation $H_{\rm Gal}^\ast(\Gamma, A)$ for the Galois cohomology of a profinite group $\Gamma$ with coefficients in a discrete $\Gamma$-module $A$. Taking $A$ to be an arbitrary finite $\mathbb ZQ$-module, consider the canonical maps
 
\begin{equation} H^\ast_{\rm Gal}(\hat{Q},A)\longrightarrow H^\ast(\hat{Q},A)\longrightarrow H^\ast(Q,A),  \end{equation}

\noindent where the second and third groups are ordinary (discrete) cohomology groups. Since every subgroup of finite index in $Q$ is open, the map $c^Q:Q\to \hat{Q}$ is an isomorphism. Hence the second map in (4.2) is an isomorphism. However, by  
[{\bf 6}, Theorem 2.10], the first map in (4.2) is also an isomorphism. As a consequence, the composition is an isomorphism; in other words, $Q$ is cohomologically ``good" in the sense of J-P. Serre [{\bf 25}, Exercises 1\&2, Chapter I.\S 2]. According to these exercises, this means that $1\rightarrow F\rightarrow \hat{G}\rightarrow \hat{Q}\rightarrow 1$ is an exact sequence of profinite groups. Therefore the residual finiteness of $Q$ implies that $G$ must be residually finite. 
\end{proof}

Our interest is primarily in the following extension of Lemma 4.11.

\begin{corollary}Let $p$ be a prime and $G$ a topological group with an open normal $\mathfrak{N}_p$-subgroup $N$. Let $F$ be a finite $p$-group. If $1\rightarrow F\rightarrow E\stackrel{\epsilon}{\rightarrow} G\rightarrow 1$ is an abstract group extension with $\epsilon^{-1}(N)$ nilpotent, then there is a unique topology on $E$ that makes it into a topological group such that   
$\epsilon: E\to G$ is a quotient map and $\epsilon^{-1}(N)$ is an open $\mathfrak{N}_p$-subgroup.
\end{corollary}

\begin{proof} Lemma 4.11 enables us to equip $M:=\epsilon^{-1}(N)$ with a unique topology that makes it a member of $\mathfrak{N}_p$ and the map $M\to N$ induced by $\epsilon$ a quotient map. 
Our plan is to show that the automorphisms of $M$ induced by conjugation in $E$ are continuous. We will then be able to endow $E$ with a topology enjoying the properties sought. Moreover, the uniqueness of this topology will follow from the uniqueness of the topology on $M$. 

Let $\alpha$ be an automorphism of $M$ arising from conjugation by an element of $E$. To verify that $\alpha$ is continuous, let $H$ be an open subgroup of $M$ and $\bar{\alpha}$ the automorphism of $M/F$ induced by $\alpha$.
Since $\epsilon(H)$ is open in $N$,  $F\alpha^{-1}(H)=\epsilon^{-1}(\bar{\alpha}^{-1}(\epsilon(H)))$ is open
in $M$. Moreover, $\alpha^{-1}(H)$ has finite index in $F\alpha^{-1}(H)$, which means that $\alpha^{-1}(H)$ is open in $F\alpha^{-1}(H)$ by Lemma 4.6(i). Thus $\alpha^{-1}(H)$ is open in $M$.
Therefore $\alpha$ is continuous, which completes the proof. 
\end{proof}

\subsection{Factoring certain topological groups}

The main result of this section is Proposition 4.15, which, next to Proposition 3.12, is the most important of the preliminary propositions required for the proof of Theorem 1.13. 

\begin{proposition} Let $p$ be a prime and $G$ a topological group containing an open normal $\mathfrak{N}_p$-subgroup $N$ such that $G/N$ is finitely generated and virtually abelian. Then $G$ possesses a subgroup $X$ such that $G=R(N)X$ and 
$N\cap X$ is compact.

\end{proposition}

The proof of Proposition 4.15 relies on the cohomological classification of group extensions, as well as the classification of extensions of modules via the functor ${\rm Ext}_R^1(\ ,\ )$. 
We will begin by establishing the proposition in the case that $R(N)$ is torsion; for this, it is only necessary to assume that $G/N$ is virtually polycyclic. Our reasoning is based on the following two
elementary lemmas, whose proofs are left to the reader. (See [{\bf 16}, 10.1.15] for a similar result.) 

\begin{lemma} Let $R$ be a ring, and let $A$ and $B$ be $R$-modules such that the underlying abelian group of $B$ is divisible. Let $0\rightarrow B\rightarrow E\rightarrow A\rightarrow 0$ be an $R$-module extension and $\xi$ the element of ${\rm Ext}_R^1(A,B)$ corresponding to this extension. If $m\cdot \xi=0$ for some $m\in \mathbb Z$, then $E$ has a submodule $X$ such that $E=B+X$ and $B\cap X=\{b\in B\ |\ m\cdot b=0\}$. \nolinebreak \hfill\(\square\)
\end{lemma}

\begin{lemma} Let $G$ be a group and $A$ a $\mathbb ZG$-module whose underlying abelian group is divisible. Let
$0\rightarrow A\rightarrow E\rightarrow G\rightarrow 1$ be a group extension giving rise to the given $\mathbb ZG$-module structure on $A$, and let $\xi$ be the corresponding element of $H^2(G,A)$. If $m\cdot \xi=0$ for some $m\in \mathbb Z$, then $E$ has a subgroup $X$ such that $E=AX$ and $A\cap X=\{a\in A\ |\ m\cdot a=0\}$. \nolinebreak \hfill\(\square\)
\end{lemma}

Equipped with these two lemmas, we dispose of the special case of Proposition 4.15. 

\begin{lemma} Let $p$ be a prime, and let $G$ be a topological group with an open normal $\mathfrak{N}_p$-subgroup $N$ such that $G/N$ is virtually polycyclic. Suppose further that $R(N)$ is torsion. Then $G$ contains a subgroup $X$ such that $G=R(N)X$ and $N\cap X$ is compact. 
\end{lemma}

\begin{proof} Set $R=R(N)$ and $Q=G/N$. First we treat the case where $R$ is infinite and every proper $G$-subgroup of $R$ is finite. 
Since $N$ is not compact,
Lemma 4.10 implies that the same holds for $N_{\rm ab}$. As a result, $R\cap N'$ is a proper subgroup of $R$, and so $R\cap N'$ is finite. From Lemmas 4.2(vii) and 4.6(iii), we deduce that $N'$ is compact. Hence it suffices to consider the case $N'=1$. 
Under this assumption, $N$ can be viewed as a $\mathbb Z_pQ$-module. Because $Q$ is virtually polycyclic,  $\mathbb Z_pQ$ is a Noetherian ring. Thus, as a finitely generated $\mathbb Z_pQ$-module, $N/R$ must have type ${\rm FP}_\infty$. This means that ${\rm Ext}^n_{\mathbb Z_pQ}(N/R,R)$ is $p$-torsion for $n\geq 0$. 
Invoking Lemma 4.16, we acquire a compact $\mathbb Z_pQ$-submodule $V$ of $N$ such that $N=R+V$ and $R\cap V$ is finite. 

We now examine the group extension $0\rightarrow A\rightarrow G/V\rightarrow Q\rightarrow 1$, where $A=N/V\cong R/(R\cap V)$.  Being a virtually polycyclic group, $Q$ is of type ${\rm FP}_\infty$. As a result, $H^n(Q,A)$ is $p$-torsion for $n\geq 0$. Consequently, by Lemma 4.17, $G/V$ has a subgroup $X^\dagger$ such that $G/V=AX^\dagger$ and $A\cap X^\dagger$ is finite. Thus, if we let $X$ be the preimage of $X^\dagger$ in $G$, then $X$ has the desired properties. This concludes the argument for the case where $R$ is infinite and every proper $G$-subgroup of $R$ is finite.

Finally, we handle the general case by inducting on $m_p(R)$. The case $m_p(R)=0$ being trivial, suppose $m_p(R)\geq 1$. Choose $K$ to be a radicable $G$-subgroup of $R$ such that $m_p(K)$ is as large as possible while still remaining less than $m_p(R)$. Then every proper $G$-subgroup of $R/K$ is finite. Hence, by the case established above, we can find a subgroup $Y$ such that $K\leq Y$, $G=RY$, and $(N\cap Y)/K$ is compact. Notice further that $K$ is the finite residual of $N\cap Y$. The inductive hypothesis furnishes, then, a subgroup $X\leq Y$ such that $N\cap X$ is compact and $Y=KX$. Since $G=RX$, this completes the proof. 
\end{proof}

The proof of the general case of Proposition 4.15 will depend heavily upon the following result.

\begin{proposition}{\rm (Kropholler, Lorensen, and Robinson [{\bf 18}, Proposition 2.1])} Let $G$ be an abelian group and $R$ a principal ideal domain such that $R/Ra$ is finite for every nonzero element $a$ of $R$. Let $A$ and $B$ be $RG$-modules that are $R$-torsion-free and have finite $R$-rank. Suppose further that $A$ fails to contain a nonzero $RG$-submodule that is isomorphic to a submodule of $B$. Then there is a positive integer $m$ such that $m\cdot {\rm Ext}_{RG}^n(A,B)=0$  
for all $n\geq 0$.  \nolinebreak \hfill\(\square\)
\end{proposition} 

\begin{proof}[Proof of Proposition 4.15]

Let $R=R(N)$ and $Q=G/N$. To begin with, we treat the case where $R$ is torsion-free and abelian, as well as simple when viewed as a $\mathbb Q_pG$-module. Suppose first that $R$ has a nontrivial compact $G$-subgroup $C$.  Then $R/C$ is torsion, so that Lemma 4.18 provides a subgroup $X$ of $G$ containing $C$ such that $G=RX$ and $(N\cap X)/C$ is compact. It follows that $N\cap X$ is compact, yielding the conclusion sought.  
Assume next that $R$ has no nontrivial compact $G$-subgroups. Since $N$ is not compact, Lemma 4.10 implies that the same holds for $N_{\rm ab}$. As a result, $R\cap N'$ is a proper subgroup of $R$. Hence the finite residual $R_0$ of $R\cap N'$ is a proper radicable $G$-subgroup of $R$. Thus $R_0=1$; that is, $R\cap N'$ is compact, and so $R\cap N'=1$. Furthermore, $N'$ is compact, which means that there is no real loss of generality in supposing $N'=1$.

Consider now the $\mathbb Z_pQ$-module extension 
$0\rightarrow R\rightarrow N\rightarrow N/R\rightarrow 0$.  
Take $Q_0$ to be a normal abelian subgroup of $Q$ with finite index. Observe that any compact $\mathbb Z_pQ_0$-submodule of $R$ must be contained in a compact $\mathbb Z_pQ$-submodule. Hence $R$ cannot possess any nonzero compact $\mathbb Z_pQ_0$-submodules. It follows, then, from
Proposition 4.19 that ${\rm Ext}^1_{\mathbb Z_pQ_0}(N/R,R)$ is torsion.  Being a vector space
 over $\mathbb Q$, this Ext-group must therefore be trivial. Thus $N$ splits as a $\mathbb Z_pQ_0$-module over $R$.  Let $V_0$ be a $\mathbb Z_pQ_0$-module complement to $R$ in $N$. Then $V_0$ is contained in a compact $\mathbb Z_pQ$-submodule $V$ of $N$.  Hence $R\cap V=0$ and $N=R+V$. 

Having obtained $V$, we shift our attention to the group extension $0\rightarrow R\rightarrow G/V\rightarrow Q\rightarrow 1$. We maintain that this extension, too, splits. To see this, notice 
$H^n(Q_0,R)\cong {\rm Ext}^n_{\mathbb Z_pQ_0}(\mathbb Z_p,R)$ for $n\geq 0$. Appealing again to Proposition 4.19, we conclude that $H^n(Q_0,R)$ is torsion for all $n\geq 0$. Hence $H^n(Q_0,R)=0$ for every $n\geq 0$, which implies $H^n(Q,R)=0$ for $n\geq 0$. This means that the desired splitting occurs; in other words, $G$ has a subgroup $X$ such that $G=NX$ and $N\cap X=V$. This subgroup, then, fulfills our requirements, completing the argument for the case where $R$ is torsion-free abelian and simple as a $\mathbb Q_pG$-module.

Now we tackle the case where $R$ is torsion-free, but not necessarily abelian. We proceed by induction on $h_p(R)$, the case $h_p(R)=0$ being trivial. Suppose $h_p(R)\geq 1$. Select $K$ to be a closed radicable $G$-subgroup of $R$ with $R/K$ abelian such that $h_p(K)$ is as large as possible while still remaining less than $h_p(R)$. Then $R/K$ is simple when regarded as a 
$\mathbb Q_pG$-module. Consequently, by the case established above, there is a subgroup $Y$ such that $K\leq Y$, $(N\cap Y)/K$ is compact, and $G=RY$. Because $K$ is the finite residual of $N\cap Y$, the inductive hypothesis provides a subgroup $X\leq Y$ such that $N\cap X$ is compact and $Y=KX$. It follows, then, that $G=RX$, thus completing the argument for the case where $R$ is torsion-free.   

Finally, we deal with the general case. Letting $T$ be the torsion subgroup of $R$, we apply the torsion-free case to $R/T$, thereby obtaining a subgroup $Y$ containing $T$ such that $G=RY$ and $(N\cap Y)/T$ is compact. Since $T$ is the finite residual of $N\cap Y$, we can apply Lemma 4.18 to $Y$, acquiring a subgroup $X\leq Y$ such that $Y=TX$ and $N\cap X$ is compact.  Then $G=RX$, so that $X$ enjoys the desired properties. 
\end{proof}

The proof of Theorem 1.13 will make use of a further factorization result for topological groups. 

\begin{proposition}  Let $p$ be a prime and $\pi$ a finite set of primes. Let $N$ be a nilpotent Hausdorff topological group with a closed normal $\mathfrak{N}_p$-subgroup $K$ such that $N/K$ is ($\pi$-torsion)-by-($\pi$-minimax). 
Then $N$ contains a ($\pi$-torsion)-by-($\pi$-minimax) subgroup $X$ such that $N=KX$.
\end{proposition}

Before proving Proposition 4.20, we establish two lemmas.

\begin{lemma} Let $\pi$ be a set of primes and $N$ be  a nilpotent group that is ($\pi$-torsion)-by-($\pi$-minimax). Then $N$ contains a finitely generated subgroup $H$ such that, for every $n\in N$, there is a $\pi$-number $m$ such that $n^m\in H$.
\end{lemma}

\begin{proof} We argue by induction on ${\rm nil}\ N$. Suppose that $N$ is abelian. Let $T$ be a $\pi$-torsion subgroup of $N$ such that $N/T$ is $\pi$-minimax, and take $\epsilon:N\to N/T$ to be the quotient map. The group $N/T$ contains a finitely generated subgroup $\bar{H}$ such that, for every $n\in N$, there is a $\pi$-number $m$ such that $\epsilon(n^m)\in \bar{H}$. Thus, if we let $H$ be a finitely generated subgroup of $N$ such that $\epsilon(H)=\bar{H}$, then $H$ can serve as the subgroup sought. 

Now we treat the case where ${\rm nil}\ N>1$. Let $Z=Z(N)$. By the abelian case of the lemma, $Z$ contains a finitely generated subgroup $A$ such that, for every $z\in Z$, there is a $\pi$-number $m$ such that $z^m\in A$.
Moreover, we can deduce from the inductive hypothesis that $N$ possesses a finitely generated subgroup $H_0$ such that, for every $n\in N$, there is a $\pi$-number $m$ such that $n^m\in ZH_0$.
Therefore, if we set $H=AH_0$, then $H$ fulfills our requirements.
\end{proof}

\begin{lemma} Let $\pi$ be a finite set of primes and $N$ a nilpotent group with a normal subgroup $K$ such that $K$ is $\pi$-radicable and $N/K$ is ($\pi$-torsion)-by-($\pi$-minimax).
Then $N$ contains a subgroup $X$ such that $N=KX$ and $X$ is ($\pi$-torsion)-by-($\pi$-minimax).
\end{lemma}

\begin{proof} We induct on ${\rm nil}_N\ K$. First suppose that $N$ centralizes $K$. Writing $Q=N/K$ and invoking Lemma 4.21, we choose a finitely generated subgroup $\bar{H}\leq Q$ such that, for every $q\in Q$, there is a $\pi$-number $m$ such that $q^m\in \bar{H}$. 
Now take $H$ to be a finitely generated subgroup of $N$ whose image in $Q$ is $\bar{H}$. Let $X$ be the {\it $\pi$-isolator} of $H$ in $N$; that is, $X$ is the subgroup of $N$ consisting of all the elements
$x$ such that $x^m\in H$ for some $\pi$-number $m$. Then Lemma 3.27 implies that $X$ is an extension of a $\pi$-torsion group by one that is $\pi$-minimax. 

We claim further that the $\pi$-radicability of $K$ ensures $N=KX$. To verify this, take $n$ to be an arbitrary element of $N$.  Then the image of $n^m$ is in $\bar{H}$ for some $\pi$-number $m$.
It follows that there exists $k\in K$ such that $kn^m\in H$.  Selecting $l\in K$ so that $l^m=k$, we have
$(ln)^m\in H$. Thus $ln\in X$, and so $n\in KX$. 

Finally, we treat the case where ${\rm nil}_N\ K>1$.   
Set $Z=Z^N(K)$.  We can suppose that $K$ is $\pi$-torsion-free. Under this assumption, $Z$ must be $\pi$-radicable.  
By the inductive hypothesis, $N$ contains a subgroup $Y$ such that $Z\leq Y$, $Y/Z$ is ($\pi$-torsion)-by-($\pi$-minimax), and $N=KY$. The base case yields, then, a ($\pi$-torsion)-by-($\pi$-minimax) subgroup $X$ of $Y$ such that $Y=ZX$. Hence $N=KX$, completing the proof.
\end{proof}

Lemma 4.22 is the basis for the proof of Proposition 4.20. 

\begin{proof}[Proof of Proposition 4.20] We argue by induction on ${\rm nil}_N\ K$. First we dispose of the case where $N$ acts trivially on $K$. Since $\mathbb Z_p/\mathbb Z$ is divisible, it follows from Proposition 4.8 that every abelian $\mathfrak{N}_p$-group
contains an abstractly finitely generated subgroup giving rise to a divisible quotient. Choose $L$ to be such a subgroup in $K$. Applying Lemma 4.22 to $N/L$, we acquire a subgroup $X$ of $N$ such that $L\leq X$, $X/L$ is ($\pi$-torsion)-by-($\pi$-minimax), and $N=KX$. Lemma 3.27, then, implies that $X$ must be ($\pi$-torsion)-by-($\pi$-minimax).

Next we handle the case where ${\rm nil}_N\ K>1$. Write $Z=Z^N(K)$. That $N$ is Hausdorff implies that $Z$ is closed; thus both $Z$ and $K/Z$ are members of $\mathfrak{N}_p$. Invoking the inductive hypothesis, we can find a subgroup $Y$ of $N$ such that $Z\leq Y$,  $Y/Z$ is ($\pi$-torsion)-by-($\pi$-minimax), and $N=KY$.  We can now invoke the case of a trivial action to obtain a ($\pi$-torsion)-by-($\pi$-minimax) subgroup $X$ of $Y$ such that $Y=ZX$. Therefore $N=KX$, which finishes the argument.
\end{proof}

\section{Tensor $p$-completion}

\subsection{$p$-Completing nilpotent groups}

Let $p$ be a prime. Throughout this section, $\mathfrak{U}_p$ will represent the class of 
topological groups whose topologically finitely generated closed subgroups are all pro-$p$ groups. Note that $\mathfrak{U}_p$ contains the class $\mathfrak{N}_p$ defined in Section 4 as a proper subclass. The embedding of a nilpotent $\mathfrak{M}$-group in an $\mathfrak{N}_p$-group employed in the proof of Theorem 1.13 will arise from a functor
$N\mapsto N_p$ from the class of nilpotent groups of finite torsion-free rank to $\mathfrak{U}_p$ 
such that $N_p$ belongs to $\mathfrak{N}_p$ if $N$ is minimax. The present section is devoted to the construction and investigation of this functor, which will coincide with classical notions in three cases:

\begin{itemize}

\item If $N$ is finitely generated, then $N_p$ is naturally isomorphic to the pro-$p$ completion of $N$.
\item If $N$ is abelian, then $N_p$ is naturally isomorphic to $N\otimes \mathbb Z_p$. 
\item If $N$ is torsion-free, then $N_p$ is naturally isomorphic to the closure of $N$ in its $p$-adic Mal'cev completion. 

\end{itemize}

Our discussion in this section will rely on the following two well-known properties of pro-$p$ completions of finitely generated nilpotent groups. 

\begin{lemma} Let $N$ be a finitely generated nilpotent group and $p$ a prime.

\begin{enumerate*}

\item If $1\to M \stackrel{\iota}{\rightarrow} N \stackrel{\epsilon}{\rightarrow} Q\to 1$ is an extension of abstract groups, then 
$1\to \hat{M_p} \stackrel{\hat{\iota_p}}{\rightarrow} \hat{N_p} \stackrel{\hat{\epsilon_p}}{\rightarrow} \hat{Q_p}\to 1$ is an extension of pro-$p$ groups. 

\item If $H\leq N$ and $\phi:H\to N$ is the inclusion monomorphism,
then $\hat{\phi_p}:\hat{H_p}\to \hat{N_p}$ is injective for every prime~$p$.  \nolinebreak \hfill\(\square\)
\end{enumerate*}

\end{lemma}

\begin{proof} That the sequence $\hat{M_p} \stackrel{\hat{\iota_p}}{\rightarrow} \hat{N_p} \stackrel{\hat{\epsilon_p}}{\rightarrow} \hat{Q_p}\to 1$ is exact holds even if $N$ is taken to be an arbitrary group. The injectivity of $\iota_p$ when $N$ is finitely generated and nilpotent follows from the cohomological property of finitely generated nilpotent groups contained in [{\bf 17}, Corollary 1.1]. Statement (ii) can be readily deduced from (i) by using the fact that every subgroup of a nilpotent group is subnormal. 
\end{proof}

Let $p$ be a prime and $N$ a nilpotent group of finite torsion-free rank. Let $\mathcal{H}$ be the set of all finitely generated subgroups of $G$. For any pair $H, K\in \mathcal{H}$ with $H\leq K$, define $\iota^{HK}:H\to K$ to be the inclusion map. Notice that, according to Lemma 5.1(ii), the map $\hat{\iota}_p^{HK}:\hat{H}_p\to \hat{K}_p$ must be injective. The {\it tensor $p$-completion} of $N$, denoted $N_p$, is defined to be the inductive limit of the groups $\hat{H}_p$ for $H\in \mathcal{H}$. Then $N_p$ is a nilpotent group with ${\rm nil}\ N_p\leq {\rm nil}\ N$.

For every $H\in \mathcal{H}$, we identify $\hat{H}_p$ with its image in $N_p$. Moreover, 
we endow $N_p$ with a topology by making $U\subseteq N_p$ open if and only if $U\cap \hat{H}_p$ is open in $\hat{H}_p$ for every $H\in \mathcal{H}$. 
In the proposition below, we show that $N_p$ is a topological group with respect to this topology.  An alternative approach to constructing $N_p$ is mentioned in Remark 5.10. 

\begin{lemma} Let $p$ be a prime and $N$ a nilpotent group of finite torsion-free rank. 

\begin{enumerate*}
\item If $H$ is a finitely generated subgroup of $N$ such that $h(H)=h(N)$, then $\hat{H}_p$ is open in $N_p$.

\item $N_p$ is a topological group that is the union of a family $\mathcal{V}$ of polycyclic pro-$p$ open subgroups of $p$-Hirsch length $h(N)$ such that, for every pair $H, K\in \mathcal{V}$, there is a subgroup $L\in \mathcal{V}$ with $H\cup K\subseteq L$.
\end{enumerate*}

\end{lemma}

\begin{proof} To prove (i), we let $K$ be an arbitrary finitely generated subgroup of $N$. Then $[K:H\cap K]$ is finite. It can thus be deduced from Lemma 5.1 that $[\hat{K}_p:\hat{H}_p\cap \hat{K}_p]<\infty$. Hence $\hat{H}_p\cap \hat{K}_p$ is open in $\hat{K}_p$. Therefore $\hat{H}_p$ is open in $N_p$.

Next we dispose of (ii). Notice that $N_p$ is the union of all the $\hat{H}_p$ such that $H$ is a finitely generated subgroup of $N$ with $h(H)=h(N)$. Hence (ii) follows from (i) and Lemma 4.12. 
\end{proof}

For any prime $p$ and nilpotent group $N$ of finite torsion-free rank, we let  $t^N_p:N\to N_p$ be the homomorphism induced by the pro-$p$ completion maps $c^H_p:H\to \hat{H}_p$ for all the finitely generated subgroups $H$ of $N$.  The map $t^N_p:N\to N_p$ 
enjoys the universal property described in Proposition 5.3 below.

\begin{proposition} Let $N$ be a nilpotent group with finite torsion-free rank and $p$ a prime. Suppose that $\phi: N\to G$ is a group homomorphism such that $G$ is a $\mathfrak{U}_p$-group.  
Then there exists a unique continuous homomorphism $\psi:N_p\to G$ such that the diagram
\begin{equation} \xymatrix{
&N \ar[d]_{t^N_p} \ar[r]^{\phi} &G\\
& N_p \ar[ur]_{\psi} &}\end{equation} commutes. 
\end{proposition}

\begin{proof} For each finitely generated subgroup $H$ of $N$, $\overline{\phi(H)}$ is a pro-$p$ group. Hence there is a unique continuous homomorphism $\psi_H:\hat{H}_p\to \overline{\phi(H)}$
such that the diagram

$$\xymatrix{
&H \ar[d]_{c^H_p} \ar[r]^{\phi} &{\overline{\phi(H)}}\\
& \hat{H}_p \ar[ur]_{\psi_H} &}$$ commutes.  The universal properties of direct limits for abstract groups and topological spaces yield, then, a unique continuous homomorphism $\psi:N_p\to G$ that makes diagram (5.1) commute.
\end{proof}

As a consequence of the preceding result, any homomorphism $\phi:N\to M$ between nilpotent groups of finite torsion-free rank induces a unique continuous homomorphism $\phi_p:N_p\to M_p$ that renders the diagram
$$\xymatrix{
&N \ar[d]_{t^N_p} \ar[r]^{\phi} &M\ar[d]_{t^M_p}\\
& N_p \ar[r]_{\phi_p} &M_p}$$
commutative. In this way, tensor $p$-completion defines a functor from the category of nilpotent groups with finite torsion-free rank to the category of $\mathfrak{U}_p$-groups. 
This functor enjoys several convenient properties, listed below.

\begin{lemma} Let $N$ be a nilpotent group with finite torsion-free rank and $p$ a prime. Then the following hold.

\begin{enumerate*}

\item If $N$ is abelian, then there is a unique topological $\mathbb Z_p$-module structure on $N_p$ that extends the topological $\mathbb Z_p$-module structures possessed by its pro-$p$ open subgroups. Moreover, the canonical $\mathbb Z_p$-module homomorphism $N\otimes \mathbb Z_p\to N_p$ is an isomorphism. 

\item The kernel of $t^N_p:N\to N_p$ is the $p'$-torsion subgroup of $N$.

\item If $1\to M \stackrel{\iota}{\rightarrow} N \stackrel{\epsilon}{\rightarrow} Q\to 1$ is an extension of abstract groups, then 
$1\to M_p \stackrel{\iota_p}{\rightarrow} N_p \stackrel{\epsilon_p}{\rightarrow} Q_p\to 1$ is an extension of topological groups.

\item If $N$ is minimax, then $N_p$ lies in the class $\mathfrak{N}_p$.

\end{enumerate*}

\end{lemma}

\begin{proof}

The first sentence of (i) follows immediately from the definition of $N_p$. The second is a consequence of the well-known fact that the canonical homomorphism $H\otimes \mathbb Z_p\to \hat{H}_p$ is an isomorphism for any finitely generated subgroup $H$ of $N$.

The case of (ii) where $N$ is finitely generated follows from the properties of the pro-$p$ completion of a nilpotent group. The general case is then an immediate consequence of this.

Next we prove (iii). By Lemma 5.1(i), if $N$ is finitely generated, then $1\to M_p \stackrel{\iota_p}{\rightarrow} N_p \stackrel{\epsilon_p}{\rightarrow} Q_p\to 1$ is an extension of pro-$p$ groups. But inductive limits commute with exact sequences of abstract groups. As a result, $1\to M_p \stackrel{\iota_p}{\rightarrow} N_p \stackrel{\epsilon_p}{\rightarrow} Q_p\to 1$ is also exact if $N$ is infinitely generated. 
Moreover, since the groups involved are locally compact, $\sigma$-compact Hausdorff groups, it follows that $1\to M_p \stackrel{\iota_p}{\rightarrow} N_p \stackrel{\epsilon_p}{\rightarrow} Q_p\to 1$  is an extension of topological groups. 

 Statement (iv) is proved by taking a series of finite length whose factors are cyclic or quasicyclic and then applying (i) and (iii).
\end{proof}

In deriving further properties of tensor $p$-completions, it will be helpful to have a notion embracing all the homomorphisms that obey the same universal property as $t^N_p$.

\begin{definition} {

\rm

 Let $G$ be an abstract group and $\bar{G}$ a topological group. A homomorphism $\tau:G\to \bar{G}$ {\it $p$-completes} $G$ if $\bar{G}$ belongs to $\mathfrak{U}_p$ and, for any homomorphism $\phi$ 
from $G$ to a $\mathfrak{U}_p$-group $H$, there is a unique continuous homomorphism $\psi:\bar{G}\to H$ such that the diagram
$$ \xymatrix{
&G \ar[d]_{\tau} \ar[r]^{\phi} &H\\
& \bar{G}\ar[ur]_{\psi} &}$$ commutes.}

\end{definition}

The next lemma is an immediate consequence of the above definition; the proof is left to the reader.

\begin{lemma} Let $N$ be a nilpotent group of finite torsion-free rank and $\bar{N}$ a topological group in $\mathfrak{U}_p$. Let $\tau:N\to \bar{N}$ be a homomorphism. If $\phi:N_p\to \bar{N}$ is the unique continuous homomorphism such that $\phi t^N_p=\tau$, then $\tau$ $p$-completes $N$ if and only if $\phi$ is an isomorphism of topological groups. \nolinebreak \hfill\(\square\)
\end{lemma}

Lemma 5.4(iii) gives rise to the following two lemmas about maps that $p$-complete.

\begin{lemma}  Let $p$ be a prime and $1\rightarrow M\rightarrow N\rightarrow Q\rightarrow 1$ be a group extension such that $N$ is nilpotent and minimax. Let $1\rightarrow \bar{M}\rightarrow \bar{N}\rightarrow \bar{Q}\rightarrow 1$ be a topological group extension such that $\bar{N}$ is nilpotent. Suppose further that these two extensions fit into a commutative diagram 

\begin{displaymath} \begin{CD}
1 @>>> M @>>> N @>>> Q@>>> 1\\
&& @VV\phi V @VV\theta V @VV\psi V&&\\
1 @>>> \bar{M}@>>> \bar{N}@>>> \bar{Q} @>>> 1.
\end{CD} \end{displaymath}

\noindent If any two of the homomorphisms $\phi$, $\theta$, and $\psi$ $p$-complete, then so does the third. 

\end{lemma}

\begin{proof} First we point out that the hypothesis implies that  $\bar{M}$, $\bar{N}$, and $\bar{Q}$ are all members of $\mathfrak{N}_p$. Appealing to Lemma 5.4(iii), we form the commutative diagram

\begin{displaymath} \begin{CD}
1 @>>> M_p @>>> N_p @>>> Q_p@>>> 1\\
&& @VV\phi' V @VV\theta' V @VV\psi' V&&\\
1 @>>> \bar{M}@>>> \bar{N}@>>> \bar{Q} @>>> 1,
\end{CD} \end{displaymath}

\noindent where $\phi'$, $\theta'$, and $\psi'$ are induced by $\phi$, $\theta$, and $\psi$, respectively. Since two of the maps $\phi'$, $\theta'$, and $\psi'$ are isomorphisms, the same holds for the third. This, then, proves the lemma.
\end{proof}

\begin{lemma} Let $N$ be a nilpotent minimax group and $p$ a prime. For any $i\geq 1$, the homomorphism $\gamma_iN\to \gamma_iN_p$ induced by $t^N_p:N\to N_p$ $p$-completes.

\end{lemma}

\begin{proof} Invoking Lemma 5.4(iii), we consider the topological group extension $1\rightarrow (\gamma_i(N))_p\rightarrow N_p\rightarrow (N/\gamma_i(N))_p\rightarrow 1$. Let $M$ be the image of  $(\gamma_i(N))_p$ in $N_p$.  Then $M$ is the closure of the image of $\gamma_i(N)$ in $N_p$. But $\gamma_i(N_p)\leq M$, and, by Lemma 4.10(i), $\gamma_i(N_p)$ is closed. Therefore $\gamma_iN_p=M$, which yields the conclusion. 
\end{proof}

\subsection{Partial tensor $p$-completion}

We conclude this section by discussing how the tensor $p$-completion functor can be applied to a normal nilpotent subgroup of finite torsion-free rank, thereby embedding the ambient group densely in a topological group with an open normal, locally pro-$p$ subgroup. This process is inspired by Hilton's notion of a {\it relative localization} from \cite{hilton}. Essential to the construction is the following result of his concerning group extensions.  

\begin{proposition}{\rm (Hilton \cite{hilton})} Let $1\rightarrow K\rightarrow G\rightarrow Q\rightarrow 1$ be a group extension and $\phi:K\to L$ a group homomorphism. Suppose further that there is an action of $G$ on $L$ that satisfies the following two properties:

{\rm (i)} $k\cdot l=\phi(k)l\phi(k^{-1})$ for all $l\in L$ and $k\in K$;

{\rm (ii)} $g\cdot \phi(k)=\phi(gkg^{-1})$ for all $g\in G$ and $k\in K$.

\noindent In addition, let $K^\dagger=\{ (\phi(k), k^{-1}) \ | \ k\in K\}\subset (L\rtimes G)$. Then $K^\dagger\lhd (L\rtimes G)$. Furthermore, if  $G^\ast=(L\rtimes G)/K^\dagger$, the diagram  

\begin{displaymath} \begin{CD}
1 @>>> K @>>> G @>>> Q @>>> 1\\
&& @VV\phi V @VV\psi V @| &&\\
1 @>>> L @>>> G^\ast @>>> Q @>>> 1
\end{CD} \end{displaymath}

\noindent commutes, where $\psi: G\to G^\ast$ is the map $g\mapsto (1, g)K^\dagger$.
\nolinebreak \hfill\(\square\)
\end{proposition}

Now assume that $G$ is a discrete group with a nilpotent normal subgroup $N$ with finite torsion-free rank, and suppose that $p$ is a prime. The functorial property of $N_p$ supplies a unique continuous action of $G$ on $N_p$ that makes $t^N_p:N\to N_p$ a $G$-group homomorphism. Furthermore, this action fulfills the condition  $n\cdot x=t^N_p(n)xt^N_p(n^{-1})$ for all $x\in N_p$ and $n\in N$. Hence, by Hilton's proposition, $N^\dagger:=\{ (t_p(n), n^{-1}) \ | \ n\in N\}$ is a normal subgroup of $N_p\rtimes G$. 
We define the {\it partial tensor $p$-completion of $G$ with respect to $N$}, denoted $G_{(N,p)}$, by 
$G_{(N, p)}=(N_p\rtimes G)/N^\dagger$. Then $G_{(N,p)}$ fits into a commutative diagram

\begin{displaymath} \begin{CD}
1 @>>> N @>>> G @>>> Q @>>> 1\\
&& @VVt^N_pV @VVV @| &&\\
1 @>>> N_p @>>> G_{(N,p)} @>>> Q @>>> 1,    
\end{CD} \end{displaymath}

\noindent and $G_{(N,p)}$ is a topological group containing an isomorphic copy of $N_p$ as an open normal subgroup.

\begin{remark} {\rm It was pointed out to the authors by an anonymous referee that both the tensor $p$-completion and the partial tensor $p$-completion can be subsumed under a single construction, using the notion of a {\it relative profinite completion} from [{\bf 27}, \S 3]. (See also the additional references provided there.) To see this, take $N$ and $G$ to be as described above and 
form the relative profinite completion $\hat{G}_H$ of $G$ with respect to a finitely generated subgroup $H$ of $N$ with maximal Hirsch length. Next, divide out by the closed subgroup of $\hat{G}_H$ generated by all the Sylow $q$-subgroups, for $q\neq p$, of the closure of the image of $N$ in $\hat{G}_H$. Then the resulting quotient can be shown to be naturally isomorphic to $G_{(N,p)}$.
} 
\end{remark}

\section{Proofs of Theorem 1.13 and its corollaries}

\subsection{Theorem 1.13}

As related in \S 2.1, the key step in our argument for Theorem 1.13 involves constructing a virtually torsion-free, locally compact cover
for $G_{(N,p)}$. For the sake of readability, we extract this step from the proof of the theorem, making it a separate proposition.

\begin{proposition} Let $p$ be a prime, and let $G$ be a topological group with an open normal $\mathfrak{N}_p$-subgroup $N$ such that
$G/N$ is finitely generated and virtually abelian. Set $S={\rm solv}(G)$. Then there is a virtually torsion-free topological group $\Gamma$ and a continuous epimorphism $\theta:\Gamma\to G$ such that the following statements hold, where $\Lambda:=\theta^{-1}(N)$.

\begin{enumerate*}

\item $\Lambda$ belongs to $\mathfrak{N}_p$.
\item ${\rm der}(\theta^{-1}(S))={\rm der}(S)$, ${\rm nil}\ \Lambda={\rm nil}\ N$, and ${\rm der}(\Lambda)={\rm der}(N)$.
\item If $G$ acts $\pi$-integrally on $N_{\rm ab}$ for some set of primes $\pi$, then $\Gamma$ acts $\pi$-integrally on $\Lambda_{\rm ab}$.
\end{enumerate*}

\end{proposition}

\begin{proof} Put $R_0=R(N)$. Also, let $P$ be the torsion subgroup of $R_0$, making $P$ a direct product of finitely many quasicyclic $p$-groups. Appealing to Proposition 4.15, we obtain a subgroup $X$ of $G$ such that $N\cap X$ is compact and $G=R_0X$.  Write $Y= X\cap N$, so that we also have the factorization $N=R_0Y$. In addition, let $X_0={\rm solv}(X)$. 

For the convenience of the reader, we divide the proof into three steps.
\vspace{5pt}

{\it Step 1: Covering $R_0$ with a radicable nilpotent torsion-free $X$-group $\Omega$.}
\vspace{3pt}

Applying Proposition 3.12 {\it ad infinitum} allows us to construct a 
sequence 

\begin{equation} \cdots \stackrel{\psi_{i+1}}{\longrightarrow}R_i\stackrel{\psi_i}{\longrightarrow} R_{i-1}\stackrel{\psi_{i-1}}{\longrightarrow}\cdots \stackrel{\psi_2}{\longrightarrow} R_1\stackrel{\psi_1}{\longrightarrow} R_0\end{equation}
 
\noindent of $X$-group epimorphisms such that the following two statements are true.

\begin{enumeratenum}

\item $R_i$ is nilpotent and radicable for each $i\geq 0$. 

\item For each $i\geq 1$, there is an $X$-group isomorphism $\nu_i: P\to \psi^{-1}_i\psi^{-1}_{i-1}\dots \psi^{-1}_1(P)$ such that $\psi_i\nu_i(u)=\nu_{i-1}(u^{p})$ for all $u\in P$ and $i\geq 2$.

\end{enumeratenum}

The three assertions below about (6.1) also hold. Statement (3) is implied by the last sentence in Proposition 3.12, and the other two follow from Lemma 3.11. 

\begin{enumeratenum}
\setcounter{enumi}{2}

\item ${\rm Ann}_{\mathbb ZX}((R_i)_{\rm ab})= {\rm Ann}_{\mathbb ZX}((R_0)_{\rm ab})$ if $i\geq 0$.

\item For every $i\geq 0$, the subgroup $Y$ of $X$ acts nilpotently on $R_i$ with ${\rm nil}_{Y}\ R_i={\rm nil}_{Y}\ R_0$ and  ${\rm der}_Y(R_i)={\rm der}_Y(R_0)$. 

\item For every $i\geq 0$, ${\rm der}_{X_0}(R_i)={\rm der}_{X_0}(R_0)$.
\end{enumeratenum}

Next we define the $X$-group $\Omega$ to be the inverse limit of the system (6.1) and $\psi$ to be the canonical map $\Omega\to R_0$. Furthermore, for each $i\geq 0$, set $\Gamma_i=R_i\rtimes X$ and $N_i=R_i\rtimes Y$.   
Corollary 4.14 implies that there is a sequence of topologies on the groups $\Gamma_i$ making each $N_i$ an open normal $\mathfrak{N}_p$-subgroup of $\Gamma_i$ and each homomorphism $\Gamma_i\to \Gamma_{i-1}$ induced by $\psi_i$ a quotient map. 
Using the topologies imparted thereby to the $R_i$, we can make $\Omega$ into a topological group by endowing it with the topology induced by the product topology. Since $X$ acts continuously on each $R_i$, the action of $X$ on $\Omega$ must also be continuous. Moreover, the subgroup $\Psi:=\psi^{-1}(P)$ is closed and isomorphic, as a topological group, to the inverse limit of the system $\cdots \rightarrow P\rightarrow P\rightarrow P$, in which every homomorphism $P\to P$ is the map $u\mapsto u^p$. As a result, $\Psi$ is isomorphic to the direct sum of finitely many copies of $\mathbb Q_p$. Since $\Omega/\Psi\cong R_0/P$ as topological groups, we conclude that $\Omega$ is a member of the class $\mathfrak{N_p}$, and that $\Omega$ is torsion-free and radicable. 
\vspace{5pt}

{\it Step 2: Defining the the cover $\Gamma$ and establishing (i) and (ii).}
\vspace{3pt}

We now set $\Gamma=\Omega\rtimes X$ and take $\theta$ to be the epimorphism from $\Gamma$ to  $G$ such that $\theta(r,x)=\psi(r)x$ for all $r\in \Omega$ and $x\in X$. 
Since $X$ acts continously on $\Omega$, $\Gamma$ is a topological group and $\theta$ is continuous. 

Writing $\Lambda=\theta^{-1}(N)$, we have $\Lambda=\Omega\rtimes Y$.   
According to Proposition 3.8(i), ${\rm nil}\ N={\rm max}\{{\rm nil}_Y\ R_0, {\rm nil}\ Y\}$.
But property (4) in Step 1 implies ${\rm nil}_Y\ \Omega={\rm nil}_Y\ R_0$. Thus, invoking Proposition 3.8(i) again,  we conclude ${\rm nil}\ \Lambda={\rm nil}\ N$ . 
Moreover, a similar argument, this time using Proposition 3.8(ii), can be adduced to show ${\rm der}(\Lambda)={\rm der}(N)$. 
 Observe further that, since $\Omega$ and $Y$ both belong to $\mathfrak{N}_p$, $\Lambda$ must be a member of $\mathfrak{N}_p$. 

We can also employ property (5) in Step 1 and Proposition 3.8(ii) to show ${\rm der}(\Sigma)= {\rm der}(S)$, where $\Sigma={\rm solv}(\Gamma)=\theta^{-1}(S)$. To accomplish this, notice $S=R_0X_0$ and $\Sigma=\Omega\rtimes X_0$. Therefore 

$${\rm der}(S)= {\rm max}\{{\rm der}_{X_0}(R_0), {\rm der}(X_0)\}= {\rm max}\{{\rm der}_{X_0}(\Omega), {\rm der}(X_0)\}={\rm der}(\Sigma).$$
\vspace{5pt}

{\it Step 3: Showing statement (iii).} 
\vspace{3pt}
 
By Lemma 3.24(ii), $X$ acts $\pi$-integrally on both $(R_0)_{\rm ab}$ and $Y_{\rm ab}$. 
It can therefore be deduced from property (3) in Step 1 that $X$ acts $\pi$-integrally on $\varprojlim (R_i)_{\rm ab}$. Now we examine the action of $X$ on $\Omega_{\rm ab}$. Observe first that the canonical maps $\Omega_{\rm ab}\to (R_i)_{\rm ab}$ give rise to a continuous $\mathbb ZX$-module homomorphism $\eta: \Omega_{\rm ab}\to \varprojlim (R_i)_{\rm ab}$.  Also, if $r\in \Omega$ such that the image of $r$ in $R_i$ lies in $R_i'$ for each $i$, then the fact that $\Omega'$ is closed in $\Omega$ ensures $r\in \Omega'$.  In other words, $\eta$ is injective. As a result, $X$ must act $\pi$-integrally on $\Omega_{\rm ab}$. Since there is a $\mathbb ZX$-module epimorphism $\Omega_{\rm ab}\oplus Y_{\rm ab}\to \Lambda_{\rm ab}$, we conclude that the action of $X$ on $\Lambda_{\rm ab}$ is $\pi$-integral. Thus $\Gamma$ acts $\pi$-integrally on $\Lambda_{\rm ab}$. 
\end{proof}

We proceed now with the proof of our main result.

\begin{proof}[Proof of Theorem 1.13] The implication ${\rm (I)}\implies {\rm (II)}$ was already established in Proposition 1.11.  
Our plan is to first show ${\rm (II)}\implies {\rm (I)}$, in the process demonstrating that the $\mathfrak{M_1^\pi}$-cover we obtain can be made to satisfy conditions (ii)-(iv) in Theorem 1.5. 
Secondly, we prove the equivalence of (II) and (III).  

For the first part of the proof, we establish Assertion A below. This statement will imply (II)$\implies$(I), together with conditions (ii)-(iv)
in Theorem 1.5. 
\vspace{5pt}

{\bf Assertion A.} {\it Let $G$ be an $\mathfrak{M^\pi}$-group with $S={\rm solv}(G)$. For any nilpotent normal subgroup $N$ of $G$ with $G/N$ finitely generated and virtually abelian and the action of $G$ on $N_{\rm ab}$ $\pi$-integral, there is an $\mathfrak{M}^\pi_1$-group $G^\ast$ and an epimorphism $\phi:G^\ast\to G$ such that the following two statements hold.
\begin{itemize}

\item $\phi^{-1}(N)$ is nilpotent of the same class and derived length as $N$.

\item ${\rm der}(S^\ast)={\rm der}(S)$, where $S^\ast={\rm solv}(G^\ast)$.

\end{itemize}}
\vspace{5pt}

 We will prove Assertion A by induction on the number of primes for which $G$ contains a quasicyclic subgroup.  If this number is zero, then we simply make $G^\ast=G$ and $\phi:G^\ast\to G$ the identity map. Suppose that this number is positive. Let $p$ be a prime for which $G$ has a quasicyclic $p$-subgroup, and let $P$ be the $p$-torsion part of $R(G)$. By the inductive hypothesis, $G/P$ can be covered by an $\mathfrak{M_1^\pi}$-group with a solvable radical of derived length ${\rm der}(S/P)$ such that $N/P$ lifts to a nilpotent subgroup of the same class and derived length. Thus, in view of Lemma 3.17, there is no real loss of generality in assuming that $G/P$ is virtually torsion-free. Lemma 3.17 also allows us to reduce to the case where the $p'$-torsion subgroup of $N$ is trivial. With this assumption, the tensor $p$-completion map $t^N_p:N\to N_p$ becomes an injection.

Form the partial tensor $p$-completion $G_{(N,p)}$, and, for convenience, identify $G$ with its image in $G_{(N,p)}$.  Put $S^\dagger={\rm solv}(G_{(N,p)})$. Since $N$ and $S$ are dense in $N_p$ and $S^\dagger$, respectively, we have ${\rm nil}\ 
 N_p={\rm nil}\ N$ and ${\rm der}(S^\dagger)={\rm der}(S)$.  Write $Q=G/N$. According to Lemmas 5.8 and 5.7, the $\mathbb ZQ$-module homomorphism $N_{\rm ab}\to (N_p)_{\rm ab}$ induced by $t^N_p:N\to N_p$ must $p$-complete. Thus ${\rm Ann}_{\mathbb ZQ}(N)\subseteq {\rm Ann}_{\mathbb ZQ}((N_p)_{\rm ab})$.  As a result, $Q$ must act $\pi$-integrally on $(N_p)_{\rm ab}$, which means that $G_{(N,p)}$ acts $\pi$-integrally on $(N_p)_{\rm ab}$.

Invoking Proposition 6.1, we obtain a virtually torsion-free topological group $\Gamma$ and a continuous epimorphism $\theta:\Gamma\to G_{(N,p)}$ such that the following statements hold.

\begin{enumeratenum}
\item $\Lambda:=\theta^{-1}(N_p)$ belongs to $\mathfrak{N}_p$.
\item ${\rm der}(\theta^{-1}(S^\dagger))={\rm der}(S)$, ${\rm nil}\ \Lambda={\rm nil}\ N$, and ${\rm der}(\Lambda)={\rm der}(N)$.
\item $\Gamma$ acts $\pi$-integrally on $\Lambda_{\rm ab}$.\
\end{enumeratenum}

We now show that $\Gamma$ contains an $\mathfrak{M_1^\pi}$-subgroup $G^\ast$ that covers $G$. 
We begin by picking elements $g_1, \dots, g_k$ of $G$ whose images generate $Q$. For each $j=1,\dots, k$, choose $g^\ast_j\in \Gamma$ such that $\theta(g^\ast_j)=g_j$.  Let $V$ be the subgroup of $\Gamma$ generated as an abstract group by $g^\ast_1,\dots, g^\ast_k$. 
Applying Proposition 4.20 to the extension 

$$1\rightarrow {\rm Ker}\ \theta \rightarrow \theta^{-1}(N)\stackrel{\theta}{\rightarrow}  N\rightarrow 1,$$

\noindent we acquire a $\pi$-minimax subgroup $H$ of $\theta^{-1}(N)$ such that $\theta(H)=N$. 
We then define $G^\ast$ to be the subgroup of $\Gamma$ generated by $H$ and $V$ as an abstract group. Notice that Proposition 3.28 implies that $G^\ast$ is $\pi$-minimax. Moreover, our selection of $H$ and $g^\ast_1,\dots, g^\ast_k$ guarantees 
$\theta(G^\ast)=G$. In addition, assertion (2) above implies that the preimage of $N$ in $G^\ast$ has the same nilpotency class and derived length as $N$, and that ${\rm der}(S^\ast)={\rm der}(S)$, where $S^\ast={\rm solv}(G^\ast)$.  
\vspace{8pt}

This now completes the proof of Assertion A. Hence (II)$\implies$(I), and the $\mathfrak{M_1^\pi}$-covering satisfies statements (ii)-(iv) from Theorem 1.5. 
\vspace{8pt}

Our final task is to prove the equivalence of (II) and (III). That (II) implies (III) is plain. Suppose that (III) holds, and write $R=R(G)$. For each $i\geq 0$, let $Z_i/R$ be the $i$th term in the upper central series of $N/R$. Since $Z_i/Z_{i-1}$ is a $\mathbb ZQ$-module that is virtually torsion-free and $\pi$-minimax {\it qua} abelian group, $Q$ acts $\pi$-integrally on $Z_i/Z_{i-1}$ for all $i\geq 1$. Now take $\bar{Z}_i$ to be the image of $Z_i$ in $N_{\rm ab}$ for each $i\geq 0$. Then, for $i\geq 1$, the action of $Q$ on $\bar{Z}_i/\bar{Z}_{i-1}$ is $\pi$-integral. In addition, the hypothesis says that $Q$ acts $\pi$-integrally on $\bar{Z}_0$. Therefore (II) is true. We have thus shown (III)$\implies$(II).
\end{proof}

\begin{remark}{\rm We point out that it seems unlikely that the kernel of the covering constructed in the proof of Theorem 1.13 can be made to be abelian in general. This is because it is not always possible to choose the subgroup $X$ in the proof of Proposition 6.1 so that the kernel of the epimorphism
$R_0\rtimes X\to G$ is abelian: that would require that $R_0\cap X$ be abelian. Of course, this still leaves open the possibility that the cover can be modified to obtain an abelian kernel. We conjecture, however, that this cannot always be accomplished, which would mean that Open Question 1.6 has a negative answer. }
\end{remark}

\subsection{Torsion-free covers}

Example 1.4 suggests that it might be feasible to strengthen Theorem \ref{thm2} to furnish a cover that is entirely torsion-free, rather than merely virtually torsion-free. This suspicion is buttressed by Lemma 3.18, which states that any finite group is a homomorphic image of a torsion-free polycyclic group. In Corollary 6.3 below, we succeed in establishing such an extension of Theorem 1.13. 
For our discussion of torsion-free covers, we use $\mathfrak{M^\pi_0}$ to represent the subclass of $\mathfrak{M^\pi_1}$ consisting of all the $\mathfrak{M^\pi_1}$-groups that are torsion-free.

\begin{corollary} Let $G$ be an $\mathfrak{M}$-group satisfying the three equivalent statements in Theorem 1.13 for a set of primes $\pi$.  Write $N={\rm Fitt}(G)$ and $S={\rm solv}(G)$.  
 Then there is an $\mathfrak{M^\pi_0}$-group $G^\ast$ and an epimorphism $\phi: G^\ast\to G$ possessing the following three properties, where $S^\ast={\rm solv}(G^\ast)$ and $N^\ast={\rm Fitt}(G^\ast)$. 

\begin{enumerate*}

\item $S^\ast=\phi^{-1}(S)$.

\item ${\rm der}(S^\ast)={\rm max}\{{\rm der}(S), 1\}$. 

\item If $S\neq 1$, then $N^\ast$ contains a normal subgroup $N^\ast_0$ of finite index such that ${\rm nil}\ N^\ast_0\leq  {\rm nil}\ N$ and ${\rm der}(N^\ast_0)\leq {\rm der}(N)$.

\end{enumerate*}

\end{corollary}

It will not have escaped the reader that, in advancing to a torsion-free cover, we have weakened the properties regarding the Fitting subgroups from Theorem 1.13. Our next corollary allows us to recover these properties, but at the expense of diluting the assertion about the derived lengths of the solvable radicals.  

\begin{corollary} Let $G$ be an $\mathfrak{M}$-group satisfying the three equivalent statements in Theorem 1.13 for a set of primes $\pi$. Set $N={\rm Fitt}(G)$ and $S={\rm solv}(G)$.   
 Then there is an $\mathfrak{M^\pi_0}$-group $G^\ast$ and an epimorphism $\phi: G^\ast\to G$ with the following three properties, where $N^\ast={\rm Fitt}(G^\ast)$ and $S^\ast={\rm solv}(G^\ast)$. 
\begin{enumerate*}

\item $N^\ast=\phi^{-1}(N)$; hence $S^\ast=\phi^{-1}(S)$.

\item If $S\neq 1$, then $S^\ast$ contains a normal subgroup $S^\ast_0$ of finite index such that $N^\ast\leq S^\ast_0$ and ${\rm der}(S^\ast_0)={\rm der}(S)$. 

\item ${\rm nil}\ N^\ast={\rm max}\{{\rm nil}\ N, 1\}$  and ${\rm der}(N^\ast)={\max}\{{\rm der}(N), 1\}$.  
\end{enumerate*}
\end{corollary}

The second open question of the article addresses whether it might be possible to completely preserve both the the nilpotency class of the Fitting subgroup and the derived length of the group when constructing a cover entirely bereft of torsion.  

\begin{openquestion2} Let $G$ be an $\mathfrak{M}$-group satisfying the three equivalent statements in Theorem 1.13 for a set of primes $\pi$.
 Is there an $\mathfrak{M^\pi_0}$-group $G^\ast$ and an epimorphism $\phi: G^\ast\to G$ satisfying conditions {\rm (ii)-(iv)} in Theorem 1.5? 
\end{openquestion2}

\noindent The proofs of Corollaries 6.2 and 6.3, provided below, reveal that the above question may be reduced to the case where $G$ is a finite group. 

We find it convenient to first prove the second corollary.

\begin{proof}[Proof of Corollary 6.4]

First we observe that the case $S=1$ is a consequence of Lemma 3.18. Assume $S\neq 1$. Let $G^\dagger$ be an $\mathfrak{M_1^\pi}$-cover of $G$ satisfying properties (ii)-(iv) in Theorem 1.5. Pick a torsion-free solvable normal subgroup $L$ in $G^\dagger$ such that $G^\dagger/L$ is finite. Writing $N^\dagger={\rm Fitt}(G^\dagger)$, we apply Lemma 3.18 to cover $G^\dagger/L$ with an $\mathfrak{M_0^{\emptyset}}$-group $H$ such that $N^\dagger/L$ lifts to a nilpotent subgroup of the same nilpotency class and derived length. Now put $G^\ast=G^\dagger\times_{G^\dagger/L} H$. Then $G^\ast$ is an $\mathfrak{M_0^\pi}$-group that covers $G$. Moreover, according to parts (ii) and (iii) of Lemma 3.17, $G^\ast$ satisfies properties (i) and (iii) of Corollary 6.4. 

It remains to establish statement (ii) of the corollary. For this, we let $\epsilon:H\to G^\dagger/L$ be the covering defined above and put $S^\dagger={\rm solv}(G^\dagger)$. Also, set $U=L\times ({\rm Ker}\ \epsilon)$. We have ${\rm der}(L)\leq {\rm der}(S)$ and

$${\rm der}({\rm Ker}\ \epsilon)\leq {\rm der}(N^\dagger/L)\leq {\rm der}(S^\dagger)={\rm der}(S).$$

\noindent As a result, ${\rm der}(U)\leq {\rm der}(S)$. Write $S^\ast={\rm solv}(G^\ast)$. Since ${\rm der}(S^\ast)\geq {\rm der}(S)$, it is straightforward to see that $S^\ast$ must therefore contain a subgroup $S_0^\ast$ such that $U\leq S^\ast_0$ and ${\rm der}(S_0)={\rm der}(S)$. 
Because $G^\ast/U$ is finite, this, then, proves assertion (iii). 
\end{proof}

\begin{proof}[Proof of Corollary 6.3]
The case $S=1$ is again a result of Lemma 3.18. Suppose $S\neq 1$. Let $G^\dagger$ and $L$ be as described in the proof of Corollary 6.4.  As before, we invoke Lemma 3.18 to obtain an epimorphism $\epsilon: H\to G^\dagger/L$ such that $H$ belongs to $\mathfrak{M_0^{\emptyset}}$; this time, however, we ensure that $\epsilon^{-1}((S^\dagger/L)^{(d-1)})$ is abelian, where $S^\dagger={\rm solv}(G^\dagger)$ and $d={\rm der}(S)$. Since
${\rm solv}(H)=\epsilon^{-1}(S^\dagger/L),$
${\rm der}({\rm solv}(H))\leq d$.
As in the previous proof, set $G^\ast=G^\dagger \times_{G^\dagger/L} H$, making $G^\ast$ an $\mathfrak{M_0^\pi}$-cover of $G$. 
Writing $S^\ast={\rm solv}(G^\ast)$, we have $S^\ast=S^\dagger\times_{G^\dagger/L} {\rm solv}(H)$. Hence assertions (i) and (ii) of the corollary are true.  

Finally, we verify that $G^\ast$ satisfies property (iii). To show this, we set $N^\ast_0=(N^\dagger\cap L)\times ({\rm Ker}\ \epsilon)$. Then $N^\ast_0$
is a normal subgroup of finite index in $N^\ast:={\rm Fitt}(G^\ast)$, and ${\rm nil}\ N^\ast_0={\rm nil}(N^\dagger\cap L)\leq {\rm nil}\ N$.  Similarly, we have ${\rm der}(N_0^\ast)\leq {\rm der}(N)$.
\end{proof}

We conclude this section by stating an important special case of Corollary 6.4.

\begin{corollary}
Let $\pi$ be a set of primes and $N$ a nilpotent $\pi$-minimax group. Then $N$ can be covered by a nilpotent torsion-free $\pi$-minimax group $N^\ast$ such that ${\rm nil}\ N^\ast= {\rm nil}\ N$ and ${\rm der}(N^\ast)={\rm der}(N)$. \hfill\(\square\)
\end{corollary}

\section{Examples of finitely generated $\mathfrak{M_{\infty}}$-groups}

In this section, we collect some examples of finitely generated 
$\mathfrak{M_\infty}$-groups.   
The principal purpose of these examples is to show that there are $\mathfrak{M_\infty}$-groups that fail to admit an $\mathfrak{M}_1$-covering whose kernel is polycyclic. This is demonstrated by Examples 7.1 and 7.6, the first of which has two primes in its spectrum and the second just a single prime. 

\begin{example}{\rm 
For any pair of distinct primes $p$ and $q$, we define a finitely generated solvable minimax group $G_1(p,q)$ with the following properties. 

\begin{itemize}
\item The spectrum of $G_1(p,q)$ is $\{p, q\}$. 
\item The finite residual $R(G_1(p,q))$ is isomorphic to $\mathbb Z(p^\infty)$, and its centralizer has infinite index. 
\item The kernel of every $\mathfrak{M_1}$-covering of $G_1(p,q)$ has $q$ in its spectrum.
\end{itemize}

Let $V$ be the group of upper triangular $3\times 3$ matrices $(a_{ij})$ such that $a_{12}\in \mathbb Z[1/pq]$, $a_{13}\in\mathbb Z[1/pq]$, $a_{23}\in\mathbb Z[1/p]$, $a_{33} = 1$, $a_{11} = q^s$, and $a_{22} = p^r$, where $r,s\in\mathbb Z$. Now let $A$ be the normal abelian subgroup of $V$ consisting of those matrices in $V$ that differ from the identity matrix at most in the $a_{13}$ entry and where $a_{13}\in\mathbb Z[1/q]$. 
Set $G_1(p,q)=V/A$. Then $G_1(p,q)$ is a finitely generated solvable minimax group, and the first two properties listed above hold. The third assertion is proved in Lemma 7.2 below.   
} 
\end{example}

\begin{acknowledgement}{\rm  Example 7.1 was suggested to the authors by an anonymous referee.}
\end{acknowledgement}

\begin{lemma}\label{eg1}
Let $p$ and $q$ be distinct primes. If $K$ is the kernel of an $\mathfrak{M_1}$-covering of $G_1(p,q)$, then $q\in {\rm spec}(K)$.
\end{lemma}

\begin{proof} Let $\phi:G^\ast\to G_1(p,q)$ be an $\mathfrak{M_1}$-covering of $G_1(p,q)$ such that ${\rm Ker}\ \phi=K$. Write $P=R(G_1(p,q)))$, $N^\ast={\rm Fitt}(G^\ast)$, 
and $P^\ast=N^\ast\cap \phi^{-1}(P)$. Since $G^\ast/N^\ast$ is virtually polycyclic, we have $\phi(P^\ast)=P$. Choose $i$ to be the largest integer such that
$\phi(Z_i(P^\ast))\neq P$, and take $g\in G^\ast$ to be a preimage under $\phi$ of the image in $G_1(p,q)$ of the $3\times 3$ matrix that differs from the identity in only the $(1,1)$ position, which is occupied by $q$. Viewing $P$ as a $\mathbb ZG^\ast$-module, we have $g\cdot a=qa$ for all $a\in P$; in other words, $g-q\in {\rm Ann}_{\mathbb ZG^\ast}(P)$.  
                                   
Let $\pi={\rm spec}(K)\cup \{p\}$. Invoking Lemma 3.21, we obtain a polynomial $f(t)\in \mathbb Z[t]$ such that the constant term of $f(t)$ is a $\pi$-number and  $f(g)\cdot a=0$ for all $a\in Z_{i+1}(P^\ast)/Z_{i}(P^\ast)$. It follows that  $f(g)$ annihilates a nontrivial $\mathbb ZG^\ast$-module quotient $\bar{P}$ of $P$.
Moreover, we can find $r\in \mathbb Z$ and a polynomial $g(t)\in \mathbb Z[t]$ such that $f(t)=(t-q)g(t)+r$. This means $ra=0$ for all $a\in \bar{P}$, and so $r=0$. As a result, $q$ divides the constant term in $f(t)$, which implies $q\in {\rm spec}(K)$.
\end{proof}

For our next family of examples, we employ the following general construction.

\begin{construct}{\rm 
Let $Q$ be a group and $A$ a $\mathbb ZQ$-module.  Denote the exterior square $A\wedge A=A\wedge_{\mathbb Z} A$ by $B$, and let $M$ be the nilpotent group with underlying set $B\times A$ and multiplication

$$(b,a)(b',a'):=(b+b'+a\wedge a', a+a').$$ 

\noindent We extend the action of $Q$ on $A$ diagonally to an action of $Q$ on $M$ and set $\Gamma=M\rtimes Q$. Then $\Gamma$ is finitely generated if and only if $Q$ is finitely generated and $A$ is finitely generated as a $\mathbb ZQ$-module.}
\end{construct}

To study our examples, we require Lemmas 7.4 and 7.5 below. In 7.4 and the rest of the paper, when $A$ is an abelian group and $R$ a commutative ring, we write $A_R$ for the $R$-module $R\otimes_{\mathbb Z} A$.  Furthermore, in statement (ii) of Lemma 7.4, we adhere to the convention ${{n}\choose{k}}:=0$ if $n<k$. 

\begin{lemma}\label{extsq}
Let $A$ be a torsion-free abelian minimax group and let $p$ be a prime.
Then the following formulae hold. 
\begin{enumerate*}
\item
$h(A)=m_p(A)+\dim_{\mathbb F_p}(A_{\mathbb F_p})$.
\item
$m_p(A\wedge A)={{h(A)}\choose{2}}-{{h(A)-m_p(A)}\choose{2}}$.
\end{enumerate*}
\end{lemma}
\begin{proof} Assertion (i) can be easily checked for the case $h(A)=1$. The general case may then be deduced by inducting on $h(A)$. 

Statement (ii) follows readily from (i):
\begin{align*}
m_p(A\wedge A)
&=h(A\wedge A)-\dim_{\mathbb F_p}(A\wedge A)_{\mathbb F_p}\\
&=\dim_{\mathbb Q}(A_{\mathbb Q}\wedge_{\mathbb Q} A_{\mathbb Q})
-\dim_{\mathbb F_p}(A_{\mathbb F_p}\wedge_{\mathbb F_p} A_{\mathbb F_p})\\
&={{h(A)}\choose{2}}-{{h(A)-m_p(A)}\choose{2}}.\qedhere\\
\end{align*}
\end{proof}

\begin{lemma}\label{onevar}
Let $a_0,\dots, a_m\in \mathbb Z$ such that $a_0a_m\ne0$ and $\gcd(a_0,\dots,a_m)=1$. Regarding $f(t):=a_mt^m+a_{m-1}t^{m-1}+\cdots +a_1t+a_0$ as an element of the Laurent polynomial ring $\mathbb Z[t,t^{-1}]$, let  $I$ be the principal ideal of $\mathbb Z[t,t^{-1}]$ generated by $f(t)$.  Furthermore,  let $\pi$ denote the set of primes dividing $a_0a_m$. Then the underlying additive group of the ring $\mathbb Z[t,t^{-1}]/I$ is torsion-free and $\pi$-minimax and has Hirsch length $m$.

\end{lemma}
\begin{proof}
Set $A=\mathbb Z[t,t^{-1}]/I$. The hypothesis $\gcd(a_0,\dots,a_m)=1$ renders $A$ torsion-free as an abelian group. As a result, we can embed $A$ in the ring $A':=\mathbb Z[\frac{1}{a_0a_m}, t, t^{-1}]/(f(t))$. It is easy to see that $A'$ is generated by the images of  $1,t,\dots,t^{m-1}$ as a $\mathbb Z[\pi^{-1}]$-module, and that these images are linearly independent over $\mathbb Z[\pi^{-1}]$. Therefore the additive group of $A$ is $\pi$-minimax with Hirsch length $m$. 
\end{proof}

The examples below appear in the first author's doctoral thesis \cite{kropholler1}, although our observation
about their coverings in Lemma 7.7 is new.

\begin{example} {\rm For each prime $p$, we define a finitely generated solvable minimax group $G_2(p)$ which will turn out to have the following properties. 

\begin{itemize}
\item The spectrum of $G_2(p)$ is $\{p\}$. 
\item The finite residual $R(G_2(p))$ is isomorphic to $\mathbb Z(p^\infty)$, and its centralizer has infinite index. 
\item The kernel of every $\mathfrak{M_1}$-covering of $G_2(p)$ has $p$ in its spectrum.
\end{itemize}

 Let $f(t)$ denote the polynomial $pt^3+t^2-t+p$ in the Laurent polynomial ring $\mathbb Z[t,t^{-1}]$. According to Lemma \ref{onevar}, $A:=\mathbb Z[t,t^{-1}]/(f(t))$ is a torsion-free $p$-minimax group of Hirsch length $3$. Moreover, ${\rm dim}_{{\mathbb F}_p}(A_{{\mathbb F}_p})=1$, so that Lemma 7.4(i) implies $m_p(A)=2$. Setting $B=A\wedge A$, we conclude from Lemma 7.4(ii) that  $m_p(B)=3$. Because $h(B)=3$, this means $B\cong \mathbb Z[1/p]^3$. 

Since $f(t)$ has no rational root, it is irreducible over $\mathbb Q$. Computing its discriminant, we ascertain that $f(t)$ must have three distinct roots in $\mathbb C$.
Taking $\alpha$, $\beta$, and $\gamma$ to be these roots, we have
$$\alpha+\beta+\gamma=-1/p;$$
$$\alpha\beta+\beta\gamma+\gamma\alpha=-1/p;$$
$$\alpha\beta\gamma=-1.$$

Now let $\phi:A_{\mathbb C}\to A_{\mathbb C}$ and $\psi:B_{\mathbb C}\to B_{\mathbb C}$ be the linear transformations induced by the left-action of $t$. Note that for $A_{\mathbb C}$ this action arises from left-multiplication, and for $B_{\mathbb C}$ it is the resulting diagonal action. Since $\alpha$, $\beta$, and $\gamma$ are the eigenvalues of $\phi$, the complex numbers $\widehat\alpha:=\beta\gamma$, $\widehat\beta:=\gamma\alpha$, and $\widehat\gamma:=\alpha\beta$ are the eigenvalues of $\psi$. 
 Moreover, we find 
$$\widehat\alpha+\widehat\beta+\widehat\gamma=-1/p;$$
$$\widehat\alpha\widehat\beta+\widehat\beta\widehat\gamma+\widehat\gamma\widehat\alpha=1/p;$$
$$\widehat\alpha\widehat\beta\widehat\gamma=1.$$
Thus the characteristic polynomial of $\psi$ is $g(t):=pt^3+t^2+t-p$. Hence $g(t)$ annihilates $B$. Notice further that $g(t)$, too, is irreducible over $\mathbb Q$. As a consequence, $\mathbb Q[t,t^{-1}]/(g(t))$ is a simple $\mathbb Q[t,t^{-1}]$-module, and every nonzero cyclic $\mathbb Z[t,t^{-1}]$-submodule of $B$ is isomorphic to $\mathbb Z[t,t^{-1}]/(g(t))$. Take $B_0$ to be one such submodule. By the same analysis as for $A$, we determine that $B_0$ is torsion-free $p$-minimax of Hirsch length $3$ with $m_p(B_0)=2$. Therefore $B/B_0\cong \mathbb Z(p^\infty)$, and, since $B_0$ is a rationally irreducible $\mathbb Z[t,t^{-1}]$-module, so is $B$. Moreover, $B_0$ is normal in the group
$\Gamma:=M\rtimes Q$ described in Construction 7.3, where $Q:=\langle t\rangle$ and the underlying set of $M$ is $B\times A$. We define $G_2(p)$ to be the quotient $\Gamma/B_0$. Then $G_2(p)$ is a finitely generated solvable
$p$-minimax group and $R(G_2(p))=B/B_0$.  Furthermore, the image of $t$ in $G_2(p)$ acts on $R(G_2(p))\cong \mathbb Z(p^\infty)$ in the same fashion as some element $\lambda$ of $\mathbb Z_p^\ast$. Since $g(\lambda)=0$, we have $\lambda\equiv -1\ \mbox{mod}\ p.$ As $\lambda\neq -1$, this means that $\lambda$ has infinite order. Consequently, 
the centralizer of $R(G_2(p))$ has infinite index in $G_2(p)$.
}
\end{example}

\begin{lemma}\label{eg2}
Let $p$ be a prime. If $K$ is the kernel of an $\mathfrak{M_1}$-covering of $G_2(p)$, then $p\in\mathrm{spec}(K)$.
\end{lemma}

To prove Lemma \ref{eg2} we employ a theorem of B. Wehrfritz. 

\begin{theorem}{\rm [{\bf 30}, Theorem 1]}
Let $\pi$ be a set of primes and $G$ an $\mathfrak{M_1^\pi}$-group. Then the holomorph $G\rtimes\mathrm{Aut}(G)$ is isomorphic to a subgroup of $\mathrm{GL}_m(\mathbb Z[\pi^{-1}])$ for some natural number $m$. \hfill\(\square\)
\end{theorem}

We also require the following special case of [{\bf 1}, Theorem 6.3]. 

\begin{theorem}{\rm (R. Baer)}
If $\pi$ is a set of primes and $N$ is a nilpotent $\pi$-minimax group, then every solvable subgroup of ${\rm Aut}(N)$ is $\pi$-minimax. \hfill\(\square\)
\end{theorem}

Baer's result has the corollary below.

\begin{corollary}
Let $\pi$ be a finite set of primes. Then every virtually solvable subgroup of $\mathrm{GL}_n(\mathbb Z[\pi^{-1}])$ is $\pi$-minimax. \hfill\(\square\)
\end{corollary}

\begin{proof}[Proof of Lemma \ref{eg2}] 
Write $G=G_2(p)$. Pick an arbitrary $\mathfrak{M_1}$-cover $G^\ast$ of $G$, and let $K$ be the kernel of the corresponding covering. In view of Lemmas 3.17 and 3.18, we will not lose any real generality if we assume that $G^\ast$ is torsion-free. Set $N={\rm Fitt}(G)$, $N^\ast={\rm Fitt}(G^\ast)$, $C=C_{G^\ast}(K)$, $N^\dagger=N^\ast\cap C$, and $K^\dagger=K\cap C$.  By Theorem 7.8 and Corollary 7.10, $G^\ast/C$ is $\pi$-minimax, where $\pi={\rm spec}(K)$. Hence $G^*/N^\dagger$ is also $\pi$-minimax.   
Let $P^\dagger$ be the intersection of  $N^\dagger$ with the preimage of the quasicyclic subgroup $P$ of $G$. Since $K^\dagger\leq Z(N^\dagger)$ and $P^\dagger/K^\dagger$ is torsion, we have $P^\dagger\leq Z(N^\dagger)$. 

Suppose $p\notin \pi$. In this case, since $N/P$ is a rationally irreducible $\mathbb ZG$-module, the covering induces a $\mathbb ZG^\ast$-module isomorphism from $N^\dagger/P^\dagger$ to a submodule of finite index in $N/P$. As in 7.6, we write $B=N/P\wedge N/P$; also, we set $B^\dagger=(N^\dagger/P^\dagger)\wedge (N^\dagger/P^\dagger)$. For the same reasons that $B\cong \mathbb Z[1/p]^3$, we have $B^\dagger\cong \mathbb Z[1/p]^3$.   

The commutator maps $N\times N\to P$ and $N^\dagger\times N^\dagger\to P^\dagger $ induce $\mathbb ZG^\ast$-module homomorphisms $\eta:B\to P$ and $\theta:B^\dagger\to P^\dagger$, respectively, with $\eta$ being surjective. Consider the commutative square

\begin{displaymath} \begin{CD} B^\dagger@>\theta>>P^\dagger\\
@VVV @VVV\\
B@>\eta>>P\end{CD} \end{displaymath}

\noindent of $\mathbb ZG^\ast$-module homomorphisms, in which both vertical maps are induced by the covering. Since the map $B^\dagger\to B$ is monic, its image must have finite index in $B$. Hence the composition $B^\dagger\to B\stackrel{\eta}{\to} P$ must be epic, which means $\theta\neq 0$. But $B^\dagger$ is a rationally irreducible $\mathbb ZG^\ast$-module because $B$ enjoys this property. Hence ${\rm Ker}\ \theta=0$.
Thus $m_p(P^\dagger)\geq 3$, and so $m_p(K^\dagger)\geq 2$, contradicting the assumption $p\notin \pi$. Therefore $p\in \pi$. 
\end{proof}

In the finitely generated $\mathfrak{M_\infty}$-groups described above, the finite residual lies in the center of the Fitting subgroup. However, it is important to recognize that this is not always the case. We illustrate this with Example 7.11, which also appears in \cite{kropholler3}, though the discussion there focuses on another aspect.  
Our description of the example is based on Construction 7.3.

\begin{example} {\rm For each prime $p$ and positive integer $m$, we define a finitely generated solvable minimax group $G_3(p,m)$ such that $R(G_3(p,m))$ is not contained in the center of the Fitting subgroup when $m\geq 2$.  
We begin by applying Construction 7.3 with $A=\mathbb Z[1/p]^{m+1}$ and $Q$ free abelian of rank two generated by the automorphisms $\phi$ and $\psi$ of $A$ with the following definitions:

$$\phi(a_0,a_1,\dots,a_m)=(a_0,a_1,a_2+a_1,a_3+a_2,\dots,a_m+a_{m-1}); $$
$$\psi(a_0,a_1,\dots,a_m)=(p^{-1}a_0,pa_1,\dots,pa_m).$$ 
If $B$, $M$, and $\Gamma$ are defined as in 7.3, then $\Gamma$ is a finitely generated solvable $p$-minimax group. Moreover, since the cyclic group $\langle \phi \rangle$ acts nilpotently on $A$ with class $m$, we have ${\rm Fitt}(\Gamma)=\langle M,\phi \rangle$. 

Letting $e_0,\dots,e_m$ be the standard basis for $A$,
take $B_0$ to be the subgroup of $B$ generated by $e_0\wedge e_1,\dots,e_0\wedge e_m$. Then $B_0$ is centralized by $\psi$ and so is normal in $\Gamma$.
We define $G_3(p,m)=\Gamma/B_0$. Observe that the action of $\langle\phi\rangle$ on $R(G_3(p,m))$ is nilpotent of class $m$, and that the subgroup of $G_3(p,m)$ generated by $R(G_3(p,m))$ and the image of $\phi$ is contained in ${\rm Fitt}(G_3(p,m))$. As a result, if $m\geq 2$, $R(G_3(p,m))$ is not contained in the center of ${\rm Fitt}(G_3(p,m))$.}
\end{example}

As observed in the introduction to the paper, part of the significance of Theorem 1.5 derives from the fact that there are uncountably many isomorphism classes of finitely generated $\mathfrak{M_\infty}$-groups. We conclude the article by using Example 1.4 to provide a simple proof of this well-known proposition. Another approach to showing this is mentioned in [{\bf 16}, p. 104]. 

\begin{proposition} There are uncountably many isomorphism classes of finitely generated $\mathfrak{M_\infty}$-groups.
\end{proposition}

The proof of the proposition is based on the following lemma. 

\begin{lemma}{\rm [{\bf 8}, Lemma III.C.42]}
Let $G$ be a finitely generated group. Then $G$ has uncountably many nonisomorphic quotients if and only if it has uncountably many normal subgroups. \nolinebreak \hfill\(\square\) 
\end{lemma}

\begin{proof}[Proof of Proposition 7.12]
In fact, we show that, for any prime $p$, there are uncountably many nonisomorphic, finitely generated $p$-minimax groups. Fix $p$ and let $\Gamma$ be the direct product of two copies of the group $G$ from Example 1.4. The center $Z(\Gamma)$ is isomorphic to $\mathbb Z(p^\infty)\oplus\mathbb Z(p^\infty)$, which has uncountably many subgroups. Therefore, by Lemma 7.13, $\Gamma$ has uncountably many nonisomorphic quotients. 
\end{proof}

\begin{acknowledgements} {\rm The authors wish to express their gratitude to four mathematicians who contributed to the development of the ideas presented here. First, we are greatly indebted to Lison Jacoboni for making us aware of the relevance of Theorem 1.5 to the investigation of random walks. Immediately upon learning of our result, she pointed out precisely how it could be employed to extend the main theorem in \cite{pittet}. We thank her for generously sharing this insight with us and giving permission for its inclusion in our paper as Corollary 1.8. 

We were also exceedingly fortunate to have had a meticulous and perspicacious  anonymous referee who read two preliminary versions of the paper carefully and provided copious comments on both the exposition and the mathematical content. His/her many valuable suggestions touched upon almost every aspect of the article, impacting its final form substantially. Of particular importance were the insights furnished by the referee concerning topological groups, as well as his/her encouraging us to seek a ``minimal" cover, which 
inspired us to prove parts (iii) and (iv) of Theorem 1.5. 

In addition, the authors are grateful to Michael Weiner for inquiring in the Penn State Topology Seminar whether our techniques could yield a complete characterization of the solvable minimax groups that occur as quotients of torsion-free solvable minimax groups. It was his question that led us to Theorem \ref{thm2}.

Lastly, we wish to thank Hector Durham for reading portions of the manuscript and correcting some errors.}

\end{acknowledgements}

\noindent{\sc 
Mathematical Sciences\\
University of Southampton\\
Highfield\\
Southampton SO17 1BJ\\
UK\\
{\it E-mail: {\tt P.H.Kropholler@southampton.ac.uk}}
\vspace{5pt}

\noindent {\sc 
Department of Mathematics and Statistics\\
Pennsylvania State University, Altoona College\\
Altoona, PA 16601\\
USA\\
{\it E-mail: {\tt kql3@psu.edu}}
}

\end{document}